\documentclass[smallextended,envcountsect,]{svjour3}

\usepackage[T1]{fontenc}

\usepackage{epstopdf}
\usepackage[caption=false]{subfig}
\usepackage{graphicx}
\usepackage{float}
\usepackage[numbers,sort&compress]{natbib}
\bibpunct[, ]{[}{]}{,}{n}{,}{,}
\makeatletter
\def\NAT@def@citea{\def\@citea{\NAT@separator}}
\makeatother
\usepackage{hyperref}
\hypersetup{
colorlinks=true,
linkcolor=blue, 
citecolor=red, 
urlcolor=blue  } 
\usepackage{mathptmx}   
\usepackage{amsmath,amssymb}

\renewcommand {\theenumi} {\rm\roman{enumi}}

\usepackage{color}
\usepackage{comment}
\usepackage[colorinlistoftodos]{todonotes}
\usepackage{tcolorbox}
\usepackage{enumitem}
\usepackage[left=4cm, right=4cm, top=4cm, bottom=4cm]{geometry}

\newcommand{\al}{\alpha}
\newcommand{\be}{\beta}
\newcommand{\ga}{\gamma}

\newcommand{\de}{\delta}
\newcommand{\eps}{\varepsilon}
\newcommand{\bx}{\bar x}

\newcommand{\iv}{^{-1} }

\newcommand {\R} {\mathbb R}
\newcommand {\N} {\mathbb N}

\newcommand {\B} {\mathbb B}
\newcommand {\Sp} {\mathbb S}
\newcommand {\dom} {{\rm dom}\,}
\newcommand {\epi} {{\rm epi}\,}

\newcommand {\sd} {\partial}

\renewcommand{\iff}{\;$\Leftrightarrow$\;\,}
\newcommand{\folgt}{\;$\Rightarrow$\;\,}

\newcommand{\vertiii}[1]{\left\vert\kern-0.25ex\left\vert\kern-0.25ex\left\vert #1\right\vert\kern-0.25ex\right\vert\kern-0.25ex\right\vert}
\newcommand{\vertiiiBig}[1]{\Big\vert\kern-0.25ex\Big\vert\kern-0.25ex\Big\vert #1\Big\vert\kern-0.25ex\Big\vert\kern-0.25ex\Big\vert}

\def\es{\emptyset}
\def\lsc{lower semicontinuous}

\def\LHS{left-hand side}
\def\RHS{right-hand side}

\def\EVP{Ekeland variational principle}
\def\Fr{Fr\'echet}

\newcommand{\abs}[1]{\left\vert#1\right\vert}

\newcommand {\diam} {{\rm diam}\,}

\newcommand{\blue}[1]{\textcolor{blue}{#1}}

\newcommand{\red}[1]{\textcolor{red}{#1}}

\newcommand{\ang}[1]{\left\langle #1 \right\rangle}
\newcommand{\qdtx}[1]{\quad\mbox{#1}\quad}
\newcommand{\AND}{\quad\mbox{and}\quad}
\newcounter{mycount}

\newcommand{\AK}[1]{\todo[inline]{AK {#1}}}


\newcommand{\EE}[3]{{\rm{\bf #1$_{#3}$}($#2$)}}
\def\sigx{\sigma_{\bold{x}}}
\def\sigxz{\sigma_{\bold{x}_0}}
\def\sigxprime{\sigma_{\bold{x}'}}
\def\sigu{\sigma_{\bold{u}}}
\def\siga{\sigma_{\bold{a}}}
\def\sigaprime{\sigma_{\bold{a}'}}
\def\Esa{\EE{E}{\sigx;\siga}{}}
\def\GEsa{\EE{GE}{\sigx;\siga}{}}
\def\Ssa{\EE{S}{\sigx;\siga}{}}
\def\Ssaal{\EE{S}{\sigx;\siga}{\al}}
\def\ASsa{\EE{AS}{\sigx;\siga}{}}
\def\ASsaal{\EE{AS}{\sigx;\siga}{\al}}

\def\Epa{\EE{E}{\bx;\siga}{}}
\def\GEpa{\EE{GE}{\bx;\siga}{}}
\def\Spa{\EE{S}{\bx;\siga}{}}
\def\Spaal{\EE{S}{\bx;\siga}{\al}}
\def\ASpa{\EE{AS}{\bx;\siga}{}}
\def\ASpaal{\EE{AS}{\bx;\siga}{\al}}

\def\Es{\EE{E}{\sigx}{}}
\def\GEs{\EE{GE}{\sigx}{}}
\def\Ss{\EE{S}{\sigx}{}}
\def\Ssxal{\EE{S}{\sigx}{\al}}
\def\ASs{\EE{AS}{\sigx}{}}
\def\ASsxal{\EE{AS}{\sigx}{\al}}

\def\Ep{\EE{E}{\bx}{}}
\def\GEp{\EE{GE}{\bx}{}}
\def\Sp{\EE{S}{\bx}{}}
\def\Spxal{\EE{S}{\bx}{\al}}

\def\ASp{\EE{AS}{\bx}{}}
\def\ASpxal{\EE{AS}{\bx}{\al}}

\newcommand{\ES}[2]{{\rm{\bf #1$_{#2}$}}}
\def\Esy{\ES{E}{\{\hat y^k\}}}
\def\GEsy{\ES{GE}{\{\hat y^k\}}}
\def\Ssy{\EE{S}{\{\bold{x}^k\}}{\{\hat y^k\}}}

\def\ASsy{\ES{S}{\{\hat y^k\}}}

\def\Assyal{\EE{AS}{\{\bold{x}^k\}}{\{\hat y^k\},\al}}

\def\GSsal{{\rm({\bf GS$_\al$})}}
\def\GSs{{\rm({\bf GS})}}

\def\PDs{{\rm({\bf PD})}}

\newcommand{\NDC}[1]{\todo[inline,color=green!40]{NDC {#1}}}
\smartqed

\begin{document}

\title{Sequential Extremal Principle: Refinements and Applications}
\author{Nguyen Duy Cuong, Alexander Kruger, Nguyen Hieu Thao}

\institute{Nguyen Duy Cuong \at
	Faculty of Mathematics,
	College of Natural Sciences\\
	 Can Tho University, Can Tho City, Vietnam\\
	 ndcuong@ctu.edu.vn
	\and
	Alexander Kruger,  Corresponding author  \at
Analytical and Algebraic Methods in Optimization Research Group \\
Faculty of Mathematics and Statistics\\
	Ton Duc Thang University, Ho Chi Minh City, Vietnam\\
alexanderkruger@tdtu.edu.vn
\and
Nguyen Hieu Thao \at School of Science, Engineering and Technology\\
 RMIT University Vietnam, Ho Chi Minh City, Vietnam\\
 thao.nguyenhieu@rmit.edu.vn
}

\dedication{Dedicated to the memory of Prof. Franco Giannessi and his great contributions to optimisation and the journal}

\date{Received: date / Accepted: date}
\maketitle
\begin{abstract}
Sequential extremality and stationarity properties are discussed with the emphasis on those corresponding to fixed sequences of translations.
Exact quantitative characterisations of the properties are provided. We show, in particular, that the sequential (as well as conventional) extremality
and approximate stationarity properties possess certain stability, while the (non-approximate)
stationarity does not.
Dual necessary conditions for the sequential extremality and stationarity properties with fixed and non-fixed sequences of translations are established.
A version of the sequential extended extremal principle is formulated.
In the statements, we employ certain generalised separation conditions \GSs\ and \GSsal\ as well as a complementary primal-dual condition \PDs.
To illustrate the model, we prove dual optimality/stationarity conditions for a constrained minimisation problem in
which the minimal value is not necessarily attained.
\end{abstract}



\section{Introduction}

We establish some refinements of the \emph{sequential extremal principle} introduced recently in \cite{CuoKru} as an extension of the conventional {extremal principle} originally proposed in \cite{KruMor80} and further developed in \cite{MorSha96,Mor06.1} and many other publications; see, e.g., \cite{Mor00,FabMor02,MorTreZhu03, Kru04,ZheNg11,BuiKru19,CuoKru25}.
Being natural generalisations of the classical \emph{separation theorem},
extremal principles (generalised separation results) are widely used in variational analysis and naturally translate into necessary
optimality conditions (multiplier rules) and various subdifferential, normal cone and coderivative calculus results in nonconvex settings; in addition to the above references, we refer the reader to \cite{Kru85.1_,Kru98,Kru03,BorZhu05,Mor06.2,MorNam22}.

The conventional extremal principle gives dual necessary conditions, which can be interpreted as generalised separation, for the extremal behaviour of a finite collection of sets near a given point in their intersection (often referred to as \emph{extremal point}).
It was shown in \cite{Kru98,Kru02} (see also \cite{Kru03,Kru04,Kru05,Kru06,Kru09,BuiKru18}) that the generalised separation conditions can be true under weaker than extremality assumptions which can be interpreted as some kinds of \emph{stationarity}.

The next definition introduces the mentioned properties.
Here and throughout the paper, we consider a collection of $n>1$ nonempty subsets $\Omega_1,\ldots,\Omega_n$ of a normed vector space $X$ and write $\{\Omega_1,\ldots,\Omega_n\}$ to denote the collection of these sets as a single object.
Symbols $\B_X$
and $B_\rho(\bx)$ denote the open unit ball and open ball with centre $\bx$ and radius $\rho>0$, respectively, while $\overline\B_X$ denotes the closed unit ball.

\begin{definition}
[Extremality, stationarity and approximate stationarity]
\label{D1.1}
Let $\bx\in\bigcap_{i=1}^n\Omega_i$.
The collection $\{\Omega_1,\ldots,\Omega_n\}$ is \begin{enumerate}
\item
\label{D1.1.1}
{extremal} at $\bx$ if there is a $\rho>0$ such that,
for any $\varepsilon>0$,
there exist $a_1,\ldots,a_n\in\eps\B_{X}$
 such that
$\bigcap_{i=1}^n(\Omega_i-a_i)\cap B_\rho(\bx)=\emptyset$;
\item
\label{D1.1.2}
{stationary} at $\bx$ if,
for any $\eps>0$,
there exist a $\rho\in(0,\eps)$
and $a_1,\ldots,a_n\in\eps\rho\B_{X}$ such that
$\bigcap_{i=1}^n(\Omega_i-a_i)\cap B_\rho(\bx)=\emptyset$;
\item
\label{D1.1.3}
{approximately stationary} at $\bx$ if,
for any $\eps>0$,
there exist a $\rho\in(0,\eps)$, $x_i\in\Omega_i\cap B_\eps(\bx)$ and $a_i\in\eps\rho\B_{X}$ $(i=1,\ldots,n)$ such that
$\bigcap_{i=1}^n(\Omega_i-x_i-a_i)\cap(\rho\B_X)
=\emptyset$.
\end{enumerate}
\end{definition}

The relationships between the properties in Definition~\ref{D1.1} are straightforward:
\eqref{D1.1.1} \folgt \eqref{D1.1.2} \folgt \eqref{D1.1.3}.
In Asplund spaces, the approximate stationarity, the weakest of the three properties, is equivalent to
the dual necessary conditions in the conventional extremal principle, i.e., the \emph{extended extremal principle} holds; cf.
\cite[Theorem~3.7]{Kru03}.

\begin{theorem}
[Extended extremal principle]
\label{T1.2}
Let $X$ be Asplund, $\Omega_1,\ldots,\Omega_n$ 
be closed, and $\bx\in\bigcap_{i=1}^n\Omega_i$.
The collection $\{\Omega_1,\ldots,\Omega_n\}$ is  approximately stationary at $\bx$ if and only if,
for any $\varepsilon>0$, there exist $x_i\in\Omega_i\cap B_\eps(\bx)$ and $x_i^*\in N^F_{\Omega_i}(x_i)$ $(i=1,\ldots,n)$ such that
$\|\sum_{i=1}^nx_i^*\|<\eps$ and
$\sum_{i=1}^n\|x_i^*\|=1$.
\end{theorem}

Recall that a Banach space is {Asplund} if every continuous convex function on an open convex set is Fr\'echet differentiable on a dense subset, or equivalently, if the dual of each its separable subspace is separable.
We refer the reader to \cite{Phe93,Mor06.1,BorZhu05} for discussions about and characterisations of Asplund spaces.
Theorem~\ref{T1.2} employs \emph{\Fr\ normal cones}.
The dual conditions are formulated in a \emph{fuzzy} form and can be interpreted as \emph{generalised separation}; cf. \cite{BuiKru19}.
The conventional extremal principle \cite{KruMor80,MorSha96,Mor06.1}, claiming that these generalised separation conditions are implied by $\{\Omega_1,\ldots,\Omega_n\}$ being extremal at $\bx$, is an immediate consequence of Theorem~\ref{T1.2}.

The extremality and stationarity concepts in Definition~\ref{D1.1} and their characterisations in the conventional and extended extremal principle as well as other typical generalisations and extensions are attached to a fixed point in the intersection of the sets (modeling an optimal or stationary point).
As observed in \cite{CuoKru}, despite its recognized versatility and numerous applications, this model does not cover an important class of optimisation problems and calculus relations `at infinity'
\cite{NguPha24,KimNguPha25,NguPha26} involving unbounded sets, where we can only have minimising (or stationary in some sense) sequences; see \cite[Examples~1.5, 2.2, 4.4 and 4.7--4.9]{CuoKru}.

It was suggested in \cite{CuoKru} to consider sequential versions of the properties in Definition~\ref{D1.1} replacing the point $\bx\in\bigcap_{i=1}^n\Omega_i$ (which may not exist) by appropriate sequences.
Below is a slightly simplified version of \cite[Definition~2.1]{CuoKru}.

\begin{definition}
[Sequential extremality, stationarity and approximate stationarity]
\label{D1.3}
The collection $\{\Omega_1,\ldots,\Omega_n\}$
is
\begin{enumerate}
\item
\label{D1.3.1}
extremal at sequences
$\{x_i^k\}\subset\Omega_i$ $(i=1,\ldots,n)$ if \begin{gather}
\label{D1.3-1}
\diam\{x_1^k,\ldots,x_n^k\}:=
\max_{1\le i,j\le n}\|x_i-x_j\|\to0 \qdtx{as}k\to+\infty,
\end{gather}
and there is a $\rho>0$ such that,
for any $\varepsilon>0$, there exist an integer $k>\eps\iv$ and points
$a_1,\ldots,a_n\in\eps\B_{X}$ such that
$\bigcap_{i=1}^n(\Omega_i-x_i^k-a_i)\cap(\rho\B_X)=\emptyset$;

\item
\label{D1.3.2}
stationary at sequences
$\{x_i^k\}\subset\Omega_i$ $(i=1,\ldots,n)$ if condition \eqref{D1.3-1} is satisfied, and, for any $\varepsilon>0$, there exist an integer $k>\eps\iv$, a $\rho\in(0,\varepsilon)$, and points
$a_1,\ldots,a_n\in\eps\rho\B_{X}$ such that
$\bigcap_{i=1}^n(\Omega_i-x_i^k-a_i)\cap(\rho\B_X)=\emptyset$;
\item
\label{D1.3.3}
approximately stationary at a sequence $\{x^k\}\subset X$ if, for any $\varepsilon>0$, there exist an integer $k>\eps\iv$, a $\rho\in(0,\varepsilon)$,
and points $x_i\in\Omega_i\cap B_\eps(x^k)$ and $a_i\in\eps\rho\B_{X}$ $(i=1,\ldots,n)$ such that $\bigcap_{i=1}^n(\Omega_i-x_i-a_i)\cap(\rho\B_X)
=\emptyset$.
\end{enumerate}
\end{definition}

Definition~\ref{D1.1} is a particular case of Definition~\ref{D1.3} with $x_i^k:=\bx$ for all $i=1,\ldots,n$ and $k\in\N$ in parts \eqref{D1.3.1} and \eqref{D1.3.2}, and $x^k:=\bx$ for all $k\in\N$ in part \eqref{D1.3.3}.
Here, the sets $\Omega_1,\ldots,\Omega_n$ are not supposed to have a common point.
Instead, condition \eqref{D1.3-1} in parts \eqref{D1.3.1} and \eqref{D1.3.2}, and conditions $x_i\in\Omega_i\cap B_\eps(x^k)$ $(i=1,\ldots,n)$ in part \eqref{D1.3.3} of Definition~\ref{D1.3} ensure
that the (generalised) distance between the sets is $0$.
The sequences in the definition may be unbounded.
Similarly to Definition~\ref{D1.1},
it holds \eqref{D1.3.1} $\Rightarrow$ \eqref{D1.3.2} in Definition~\ref{D1.3} and, if $\{\Omega_1,\ldots,\Omega_n\}$ is stationary at sequences
$\{x_i^k\}\subset\Omega_i$ $(i=1,\ldots,n)$, then, for each $i\in\{1,\ldots,n\}$, it is approximately stationary at the sequence $\{x_{i}^k\}$.

Employing the generalised separation result from \cite[Theorem~3.1]{CuoKru25} (proved using standard techniques based on the Ekeland variational principle and subdifferential sum rules), dual necessary conditions for the sequential approximate stationarity (and for the more advanced sequential approximate $\al$-stationarity) were established in \cite{CuoKru}.
These conditions preserve the generalised separation pattern typical of the conventional extremal principle.
In the Asplund space setting, they lead to the following statement generalising and extending Theorem~\ref{T1.2}; cf. \cite[Corollary~3.7]{CuoKru}.

\begin{theorem}
[Sequential extended extremal principle]
\label{T1.4}
Let $X$ be Asplund, $\Omega_1,\ldots,\Omega_n$ 
be closed, and $\{x^k\}\subset X$.
The collection
$\{\Omega_1,\ldots,\Omega_n\}$
is approximately stationary at
{$\{x^k\}$} if and only if,
for any $\varepsilon>0$, there exist an integer $k>\eps\iv$, and points
$x_i\in\Omega_i\cap B_\eps(x^k)$ and $x_i^*\in N^F_{\Omega}(x_i)$ $(i=1,\ldots,n)$ such that
$\|\sum_{i=1}^nx_i^*\|<\eps$ and
$\sum_{i=1}^n\|x_i^*\|=1$.
\end{theorem}

\begin{remark}
Compared to Theorem~\ref{T1.4}, the statement of
\cite[Corollary~3.7]{CuoKru} contains an additional dual characterisation of sequential approximate stationarity of $\{\Omega_1,\ldots,\Omega_n\}$ at {$\{x^k\}$}:
\emph{for any $\varepsilon>0$ and $\tau\in(0,1)$, there exist an integer $k>\eps\iv$, and points $x_0\in\eps\B_X$,
$x_i,x'_i\in\Omega_i\cap B_\eps(x^k)$, $a_i\in\eps\B_{X}$, $x_i^*\in N^F_{\Omega}(x_i)$ $(i=1,\ldots,n)$ such that
$\|\sum_{i=1}^nx_i^*\|<\eps$,
$\sum_{i=1}^n\|x_i^*\|=1$, and}
\begin{gather*}
\sum_{i=1}^n\ang{x_i^*,x_0+a_i+x'_i-x_i}
>\tau\max_{1\le i\le n}\|x_0+a_i+x'_i-x_i\|.
\end{gather*}
The origins of this type of conditions can be traced back to the \emph{unified separation theorem} by Zheng \& Ng \cite{ZheNg11}.

The above characterisation obviously implies the one in Theorem~\ref{T1.4}.
At the same time, it is a straightforward consequence of the one in Theorem~\ref{T1.4}.
Indeed, let $\eps>0$, $k\in\N$, $x_i\in\Omega_i\cap B_\varepsilon(x^k)$, $x_i^*\in N^F_{\Omega_i}(x_i)$ $(i=1,\ldots,n)$, $\|\sum_{i=1}^nx_i^*\|<\eps$ and
$\sum_{i=1}^n\|x_i^*\|=1$.
Let $\tau\in(0,1)$ and $\tau'\in(\tau,1)$.
For each $i=1,\ldots,n$, there exists an $a'_i\in X$ such that $\|a'_i\|=1$ and $\langle x_i^*,a'_i\rangle\ge \tau'\|x_i^*\|$.
Choose an $\eps'\in(0,\varepsilon)$ and define
$a_i:=\eps' a'_i$.
Then $\|a_i\|=\eps'<\eps$ $(i=1,\ldots,n)$, and
\begin{align*}
\sum_{i=1}^n \langle x_i^*,a_i\rangle
\ge\eps'\tau' \sum_{i=1}^n \|x_i^*\|
=\eps'\tau'
>\tau \max_{1\le i\le n}\|a_i\|.
\end{align*}
Thus, the characterisation formulated above holds with $x_0:=0$ and $x_i':=x_i$ $(i=1,\ldots,n)$.
\end{remark}

The sequential model developed in \cite{CuoKru} extends the applicability of the approach based on
extremal principles.
In particular, it leads to natural dual optimality conditions (multiplier rules) for minimising sequences; see \cite[Section~5]{CuoKru}.
This sequential model is more flexible than the model proposed in \cite{NguPha24,KimNguPha25,NguPha26} and does not require introducing an artificial infinity point.

In this paper, we extend and improve Definition~\ref{D1.3} along several lines.
\begin{enumerate}
\item
The model adopted in the current paper does not assume condition \eqref{D1.3-1} making an essential part of the properties in parts \eqref{D1.3.1} and \eqref{D1.3.2} of the definition, and implicitly present also in part \eqref{D1.3.3}.
This natural condition is satisfied in many settings involving unbounded sequences modeling optimisation problems `at infinity' as discussed in
\cite{CuoKru}.
At the same time, it is superfluous when proving dual necessary conditions.
In the conventional `at-a-point' setting, this fact was observed in \cite[Section~4.2]{BuiKru18}.
\item
The properties in Definition~\ref{D1.3} are defined for the given sequences
$\{x_i^k\}\subset\Omega_i$ $(i=1,\ldots,n)$ in parts \eqref{D1.3.1} and \eqref{D1.3.2}, and $\{x^k\}\subset X$ in part \eqref{D1.3.3}.
However, a closer look at the definition shows that
the properties are actually determined by some subsequences of the given sequences.
The definitions adopted in the current paper
are more specific in identifying the sequences relevant for the properties.
\item
The last property in Definition~\ref{D1.3} is defined using a slightly different language compared to the first two properties.
This complicates the comparison of the definitions
and results.
One can see this already when reading the above items.
An advantage of the model adopted in the current paper is the uniformity of the presentation of all three groups of the studied properties: extremality, stationarity and approximate stationarity, simplifying the comparison of the definitions and results.
\item
All three properties in Definition~\ref{D1.3} (as well as Definition~\ref{D1.1}) implicitly assume the existence of certain sequences of translations
converging to zero.
These sequences play an important role in the definitions.
In the current paper, we make them explicit and
study the properties corresponding to fixed sequences of translations.
One can go further and consider more general than translations deformations of the given sets; cf. \cite{CuoKruTha24,CuoKruTha25}.
\item
Unlike the descriptive Definition~\ref{D1.3}, the corresponding properties are defined in this paper using exact quantitative estimates of the properties; see Definition~\ref{D2.01}.
\end{enumerate}

These modifications lead to a clearer model resulting in improved sequential dual necessary conditions with some additional restrictions on the dual variables and also shed some new light on
extremality/stationarity settings, especially on the role of the vectors/sequences of translations.

The new definitions are introduced and discussed in Section~\ref{S2}.
It consists of two subsections.
In the first one, we discuss the properties corresponding to fixed sequences of translations $\{a^k_i\}\subset X$ $(i=1,\ldots,n)$.
This line of research is new even in the conventional `at-a-point' setting.
The exact quantitative characterisations of the respective properties are provided.
We show, in particular, that the sequential (as well as conventional) extremality and approximate stationarity properties possess certain stability, while the (non-approximate) stationarity does not.
The established relations are then used in the second subsection when proving certain facts for the properties defined \emph{for some} sequences $\{a^k_i\}$ $(i=1,\ldots,n)$.
The conventional setting when the sequences
$\{x_i^k\}\subset\Omega_i$ $(i=1,\ldots,n)$ reduce to a single point in the intersection of the sets is given a special attention.

In Section~\ref{S3}, we establish dual necessary
conditions for the sequential extremality and stationarity properties discussed in Section~\ref{S2} and a more elaborate version of the sequential extended extremal principle.
In the statements, we employ certain generalised separation conditions \GSs\ and \GSsal\ as well as a complementary primal-dual condition \PDs\ providing additional restrictions on the associated dual vectors.
When proving the dual necessary conditions, we employ a simplified version of \cite[Corollary~3.2]{CuoKru25} instead of the general separation result \cite[Theorem~3.1]{CuoKru25} used in \cite{CuoKru}, producing a straightforward proof.

To illustrate the model, we consider in Section~\ref{S4} a constrained minimisation problem in which the minimal value is not necessarily attained and deduce stronger than in \cite{CuoKru} sequential optimality and stationarity conditions for more general
types of stationary sequences.
Some concluding remarks are collected in Section~\ref{conclusions}.

\subsubsection*{Preliminaries}
\label{Pre}

Our basic notation is standard; see, e.g., \cite{BorZhu05,Mor06.1}.
Throughout the paper, $X$ is a normed space (typically Banach or Asplund).
Its topological dual is denoted by $X^*$, while $\langle\cdot,\cdot\rangle$ denotes the bilinear form defining the pairing between the two spaces.
The product spaces $X^n$ and $X\times\R$
are assumed to be equipped with the maximum norms.
We use the same symbols $\|\cdot\|$ and $d(\cdot,\cdot)$ to denote norms and distances (including point-to-set distances with the convention $d(x,\emptyset)=+\infty$ for any $x$) in all (primal and dual) spaces.
The results of the current paper can be easily extended to the setting of general product norms satisfying certain compatibility conditions with the norm on $X$ (particularly, $\ell^p$ norms).
This can only affect some quantitative estimates.
We refer the readers to \cite{CuoKru,CuoKru25,Cuo26} for the details.
Note that the counterparts of Definition~\ref{D1.3} and Theorem~\ref{T1.4} are formulated in \cite{CuoKru} using general product norms.

Symbols $\R$, $\R_+$ and $\N$ stand for the sets of all, respectively, real, nonnegative real and positive integer numbers.
If $\al\in\R$, then $\al_+:=\max\{\al,0\}$ is the \emph{positive part} of $\al$.
The notation
$\{x^k\}\subset\Omega$ denotes a sequence of points $x^k\in\Omega$ $(k\in\N)$.
We sometimes use a shorter notation $\sigma_x:=\{x^k\}$.

\paragraph*{Normal cones and subdifferentials.}
We first recall the definitions of normal cones and subdifferentials in the sense of Fr\'echet and Clarke;
see, e.g., \cite{Cla83,Kru03,Mor06.1}.
Given a subset $\Omega$ of a normed space $X$ and a point $\bx\in \Omega$, the sets
\begin{gather}\label{NC}
N_{\Omega}^F(\bx):= \Big\{x^\ast\in X^\ast\mid
\limsup_{\Omega\ni x{\rightarrow}\bar x,\;x\ne\bx} \frac {\langle x^\ast,x-\bx\rangle}
{\|x-\bx\|} \le 0 \Big\},
\\\label{NCC}
N_{\Omega}^C(\bx):= \left\{x^\ast\in X^\ast\mid
\ang{x^\ast,z}\le0
\;\;
\text{for all}
\;\;
z\in T_{\Omega}^C(\bx)\right\}
\end{gather}
are the, respectively, \emph{Fr\'echet} and \emph{Clarke normal cones} to $\Omega$ at $\bx$.
Symbol $T_{\Omega}^C(\bx)$
in \eqref{NCC}
stands for the \emph{Clarke tangent cone} to $\Omega$ at $\bx$.
The sets \eqref{NC} and \eqref{NCC} are nonempty
closed convex cones satisfying $N_{\Omega}^F(\bx)\subset N_{\Omega}^C(\bx)$.
If $\Omega$ is a convex set, they reduce to the normal cone $N_{\Omega}(\bx)$ in the sense of convex analysis.

For an extended-real-valued function $f:X\to\R_\infty:=\R\cup\{+\infty\}$ on a normed space $X$ with epigraph $\epi f:=\{(x,\alpha) \in X \times \mathbb{R}\mid f(x) \le \alpha\}$,
the \emph{Fr\'echet} and \emph{Clarke subdifferentials} of $f$ at $\bar x\in\dom f:=\{x \in X\mid f(x) < +\infty\}$
are defined, respectively, by
\begin{gather}\label{SF}
\sd^F f(\bar x):= \left\{x^* \in X^*\mid (x^*,-1) \in N^F_{\epi f}(\bar x,f(\bar x))\right\},
\\ \label{SC}
\partial^C{f}(\bx):= \left\{x^\ast\in X^\ast\mid
(x^*,-1)\in N_{\epi f}^C(\bx,f(\bx))\right\}.
\end{gather}
The sets \eqref{SF} and \eqref{SC} are closed and convex, and satisfy
$\partial^F{f}(\bx)\subset\partial^C{f}(\bx)$.
If $f$ is convex, they
reduce to the subdifferential $\sd f(\bx)$ in the sense of convex analysis.


We often use the generic notations $N$ and $\sd$ with the convention that $N:=N^C$ and $\sd:=\sd^C$ if $X$ is a general
normed space, and $N:=N^F$ and $\sd:=\sd^F$ if $X$ is Asplund.

\paragraph*{Generalised separation.}

The proof of our main result is based on the next generalised separation statement.
It is a simplified version of a more general result from \cite[Corollary 3.2]{CuoKru25} extending the \emph{unified separation theorems} by Zheng \& Ng \cite{ZheNg05.2,ZheNg11} and their slightly more advanced versions in \cite
[Lemma~2.1]
{CuoKruTha24}.

\begin{theorem}
[Generalised separation]
\label{T1.5}
Let $X$ be Banach, $\Omega_1,\ldots,\Omega_n$ 
be closed, $x_{0i}\in\Omega_i$ ($i=1,\ldots,n$),
$\eps>0$, $\de>0$ and $\rho>0$.
Suppose that
$\bigcap_{i=1}^n\Omega_i\cap(\rho\overline\B_X)=\emptyset$, and $\|x_{0i}\|<\varepsilon$ ($i=1,\ldots,n$).
The following assertions hold true:
\begin{enumerate}
\item
\label{T1.5-1}
there exist $x_i\in\Omega_i\cap B_\de(x_{0i})$, $x_i^*\in X^*$ ($i=1,\ldots,n$) and {$v\in\rho\B_X$} such that
\begin{gather}
\label{T1.5-4}
{\de}\sum_{i=1}^{n} d\left(x_i^*,N_{\Omega_i}(x_i)\right)+ \rho\Big\|\sum_{i=1}^nx_i^*\Big\| <{\eps},\quad \sum_{i=1}^{n}\|x_i^*\|=1,
\\
\notag
\sum_{i=1}^{n} \langle x_i^*, v-x_i\rangle=\max_{1\le i\le n}\|v-x_i\|;
\end{gather}

\item
\label{T1.5-2}
if $X$ is Asplund, then, for any $\tau\in(0,1)$, there exist $x_i\in\Omega_i\cap B_\de(x_{0i})$, $x_i^*\in X^*$ ($i=1,\ldots,n$) and {$v\in\rho\B_X$} such that condition \eqref{T1.5-4} is satisfied, and
\begin{gather*}
\sum_{i=1}^{n} \langle x_i^*, v-x_i\rangle>\tau\max_{1\le i\le n}\|v-x_i\|.
\end{gather*}
\end{enumerate}
\end{theorem}

In view of our convention, Theorem~\ref{T1.5} employs {Clarke normal cones} in part \eqref{T1.5-1} and {\Fr\ normal cones} in part~\eqref{T1.5-2}.
Similar to the conventional extremal principle, it is a consequence of the {\EVP} and corresponding subdifferential sum rules.


\begin{remark}
By \cite[Corollary 3.2]{CuoKru25},
the conclusions of Theorem~\ref{T1.5} hold true under a weaker than $\|x_{0i}\|<\varepsilon$ ($i=1,\ldots,n$) condition
\begin{gather*}
\max_{1\le i\le n}\|x_{0i}\| <\inf_{\substack{u_i\in\Omega_i\;(i=1,\ldots,n)}}\; \inf_{u\in\rho\B_X}\max_{1\le i\le n} \|u_i-u\|+\varepsilon.
\end{gather*}
The second (inner) infimum in the \RHS\ of the inequality represents a convex minimisation problem of finding the \emph{Chebyshev centre} of the set of points $\{u_1,\ldots,u_n\}$; see, e.g., \cite{BoyVan04}.
If an analytical solution of the problem can be found, the condition can be simplified; see~\cite{Cuong3}.
\end{remark}

\section{Sequential extremality, stationarity and approximate stationarity}
\label{S2}

In this section, we
refine the sequential extremality and stationarity concepts in Definition~\ref{D1.3}.

\subsection{Relative sequential extremality, stationarity and approximate stationarity}
\label{S2.01}

We first discuss more specific than in Definition~\ref{D1.3} extremality and stationarity properties of $\{\Omega_1,\ldots,\Omega_n\}$ at a given sequence of points
$\bold{x}^k:=(x^k_1,\ldots,x^k_n)\in\widehat\Omega :=\Omega_1\times\ldots\times\Omega_n$ $(k\in\N)$ \emph{with respect to} a given sequence of translations $\bold{a}^k:=(a^k_1,\ldots,a^k_n)\in X^n\setminus\{0\}$ $(k\in\N)$ such that $\bold{a}^k\to0$ as $k\to+\infty$.
We are going to use for these sequences the brief notations $\sigx:=\{\bold{x}^k\}$ and $\siga:=\{\bold{a}^k\}$, respectively, and similar notations for some other sequences.

Given an $\bold{x}:=(x_1,\ldots,x_n)\in X^n$, denote $\zeta(\bold{x}) :=d\big(0,\bigcap_{i=1}^n(\Omega_i-x_i)\big)$.

The next definition introduces the key properties studied in the paper together with their notations which are used throughout the paper for brevity.
Here and in what follows, {\bf E}, {\bf S} and {\bf AS} stand for the extremality, stationarity and approximate stationarity, respectively.

\begin{definition}
[Relative sequential extremality, stationarity and approximate stationarity]
\label{D2.01}
The collection $\{\Omega_1,\ldots,\Omega_n\}$ is
\begin{enumerate}
\item
\label{D2.01.1}
extremal at $\sigx$
with respect to $\siga$ (denoted \Esa) if
\begin{gather}
\label{D2.01-1}
\rho(\sigx;\siga):= \liminf_{k\to+\infty}
\zeta(\bold{x}^k+\bold{a}^k)>0;
\end{gather}
if $\rho(\sigx;\siga)=+\infty$,
we say that $\{\Omega_1,\ldots,\Omega_n\}$ is globally extremal at $\sigx$ with respect to $\siga$ (denoted \GEsa);
\item
\label{D2.01.2}
$\al$-stationary at $\sigx$ with respect to $\siga$ (denoted \Ssaal) if
\begin{gather}
\label{D2.01-2}
\al(\sigx;\siga):= \limsup_{\substack{k\to+\infty}}
\frac{\|\bold{a}^k\|} {\zeta(\bold{x}^k+\bold{a}^k)} <\al<+\infty;
\end{gather}
if $\al(\sigx;\siga)=0$ (hence, \eqref{D2.01-2} is satisfied with any $\al>0$), we say that $\{\Omega_1,\ldots,\Omega_n\}$ is stationary at $\sigx$ with respect to $\siga$ (denoted \Ssa);
\item
\label{D2.01.3}
approximately $\al$-stationary at $\sigx$ with respect to $\siga$ (denoted \ASsaal)
if
\begin{gather}
\label{D2.01-3}
\widetilde\al(\sigx;\siga):= \inf_{
\sigma_{\bold{u}}:=\{\bold{u}^k\}\subset\widehat\Omega,\,\bold{u}^k-\bold{x}^k\to0}
\al(\sigma_{\bold{u}};\siga)<\al<+\infty;
\end{gather}
if $\widetilde\al(\sigx;\siga)=0$ (hence, \eqref{D2.01-3} is satisfied with any $\al>0$), we say that $\{\Omega_1,\ldots,\Omega_n\}$ is approximately stationary at $\sigx$ with respect to $\siga$ (denoted \ASsa).
\end{enumerate}
\end{definition}

The properties in Definition~\ref{D2.01} are defined for an arbitrary sequence $\sigx\subset\widehat\Omega$, possibly unbounded.
At the same time, the model studied in this paper covers the conventional `at-a-point' setting.
In the particular case when $x_i^k=\bx$ for some $\bx\in\bigcap_{i=1}^n\Omega_i$ and all $i=1,\ldots,n$ and $k=1,2,\ldots$ (hence, $\zeta(\bold{x}^k+\bold{a}^k) =d(\bx,\bigcap_{i=1}^n(\Omega_i-a_i))$), we write
$\rho(\bx;\siga)$, $\al(\bx;\siga)$ and $\widetilde\al(\bx;\siga)$ instead of $\rho(\sigx;\siga)$, $\al(\sigx;\siga)$ and $\widetilde\al(\sigx;\siga)$, respectively, call the corresponding properties extremality, global extremality, $\al$-sta\-tionarity, stationarity, approximate $\al$-sta\-tionarity and approximate stationarity at $\bx$ with respect to $\siga$, and employ the following respective notations: \Epa, \GEpa, \Spaal, \Spa, \ASpaal\ and \ASpa.

\begin{remark}
\label{R2.02}
\begin{enumerate}
\item
The numbers $\rho(\sigx;\siga)$, $\al(\sigx;\siga)$ and $\widetilde\al(\sigx;\siga)$ computed in \eqref{D2.01-1}, \eqref{D2.01-2} and \eqref{D2.01-3} give exact quantitative characterisations of the respective properties, measuring their `power'.
\item
\label{R2.02.2}
Our standing assumption that
 $\bold{a}^k\ne0$ $(k\in\N)$ ensures that the fraction in \eqref{D2.01-2} is well-defined.
In view of \eqref{D2.01-1}, in the case of extremality we automatically have $\bold{a}^k\ne0$ for all sufficiently large $k\in\N$.
\item
\label{R2.02.3}
The extremality property in part \eqref{D2.01.1} of Definition~\ref{D2.01} is fully determined by the sequence $\sigma_{\bold{y}}:=\{\bold{y}^k\}\subset X^n$,
 where
$\bold{y}^k:=\bold{x}^k+\bold{a}^k$ $(k\in\N)$ and does not depend on its representation as the sum of two sequences.
So, it is possible to talk about extremality `with respect to $\sigma_{\bold{y}}$'.
We adopt the terminology in the definition for consistency with that used in parts \eqref{D2.01.2} and \eqref{D2.01.3}.
\item
\label{R2.02.4}
Unlike the extremality, the stationarity properties in parts \eqref{D2.01.2} and \eqref{D2.01.3} of Definition~\ref{D2.01} depend on both sequences $\sigx$ and $\siga$.
At the same time, the `$\al$-stationarity with respect to $\sigma_{\bold{y}}$',
 defined by the condition
\begin{gather*}
\al(\sigma_{\bold{y}}):= \limsup_{\substack{k\to+\infty}}
\frac{d(\bold{y}^k,\widehat\Omega)} {\zeta(\bold{y}^k)} <\al<+\infty
\end{gather*}
for a sequence $\sigma_{\bold{y}}:=\{\bold{y}^k\}\subset X^n$, $\bold{y}^k:=(y^k_1,\ldots,y^k_n)$ satisfying $d(\bold{y}^k,\widehat\Omega)>0$ $(k\in\N)$ and
$d(\bold{y}^k,\widehat\Omega)\to0$ as $k\to+\infty$, can also be of interest.
If $\bold{y}^k=\bold{x}^k+\bold{a}^k$, then obviously $\al(\sigma_{\bold{y}}) \le\al(\sigx;\siga)$.
Hence, this property is in general weaker than the one in part \eqref{D2.01.2} of Definition~\ref{D2.01}.
Moreover, given a $\sigma_{\bold{y}}\subset X^n$ as above, one can always find appropriate sequences $\sigx:=\{\bold{x}^k\}\subset\widehat\Omega$ and $\siga:=\{\bold{a}^k\}\subset X^n$ such that $\bold{y}^k:=\bold{x}^k+\bold{a}^k$ $(k\in\N)$ and $\al(\sigma_{\bold{y}}) =\al(\sigx;\siga)$.
For any $k\in\N$, we have
\begin{gather*}
\zeta(\bold{y}^k)\ge\max_{1\le i\le n} d(y^k_i,\Omega_i)) =d(\bold{y}^k,\widehat\Omega).
\end{gather*}
It follows from the above estimate that $\al(\sigma_{\bold{y}})\le1$, and the property is automatically satisfied with any $\al>1$.
Hence,
it only makes sense to study this property for small $\al>0$.
\item
The sequence $\sigma_{\bold{u}}$
in part \eqref{D2.01.3} of Definition~\ref{D2.01} can be regarded as an approximation of the given sequence $\sigx$,  justifying the name of the property.
\end{enumerate}
\end{remark}

The next two propositions are immediate consequences of Definition~\ref{D2.01}.

\begin{proposition}
{\rm(i)} If $\rho(\sigx;\siga)>0$, then $\al(\sigx;\siga)=0$;\quad
{\rm(ii)} $\widetilde\al(\sigx;\siga) \le\al(\sigx;\siga)$.

As a consequence,
\GEsa\ \folgt \Esa\ \folgt \Ssa\ \folgt \ASsa,
and\\ \Ssaal\ \folgt \ASsaal\; for any $\al>0$.

In particular, if $\bx\in\bigcap_{i=1}^n\Omega_i$, then
\GEpa\ \folgt \Epa\ \folgt \Spa\ \folgt \ASpa,
{and} \Spaal\ \folgt \ASpaal\; for any $\al>0$.
\end{proposition}

\begin{proposition}
Let
$\{\bold{x}'^k,\bold{a}'^k\}$ be a subsequence of $\{\bold{x}^k,\bold{a}^k\}$, $\sigma_{\bold{x}'}:=\{\bold{x}'^k\}$ and $\sigma_{\bold{a}'}:=\{\bold{a}'^k\}$.
Then
$\rho(\sigma_{\bold{x}'};\sigma_{\bold{a}'}) \ge\rho(\sigx;\siga)$, $\al(\sigma_{\bold{x}'};\sigma_{\bold{a}'}) \le\al(\sigx;\siga)$ and $\widetilde\al(\sigma_{\bold{x}'};\sigma_{\bold{a}'}) \le\widetilde\al(\sigx;\siga)$.

As a consequence,
\EE{P}{\sigx;\siga}{} \folgt \EE{P}{\sigxprime;\sigaprime}{},
where {\rm\textbf{P}} stands for any of the properties: {\rm\textbf{E},~\textbf{GE}, \textbf{S$_\al$}, \textbf{S}, \textbf{AS$_\al$} or \textbf{AS}}.
In particular, if $\bx\in\bigcap_{i=1}^n\Omega_i$, then
\EE{P}{\bx;\siga}{} \folgt \EE{P}{\bx;\sigaprime}{}.
\end{proposition}


We now illustrate the properties in Definition~\ref{D2.01} by examples.
The first one emphasises the role of the sequence $\siga$ of translations.
We aim at providing some insights on how to choose appropriate translations ensuring the desired extremality or stationarity properties.

\begin{example}
\label{E2.5}
Let $X:=\R^2$,
$\Omega_1:=\{(u,v)\in X\mid v\ge0\}$, $\Omega_2:=\{(u,v)\in X\mid v\le0\}$ and $\bx:=(0,0)$.
The pair $\{\Omega_1,\Omega_2\}$ is obviously extremal
at $\bx$ in the sense of Definition~\ref{D1.1}\eqref{D1.1.1}: for any $\eps>0$, one can take $a_1:=(0,0)$ and $a_2:=(0,\eps/2)$.

We next show that the extremality and stationarity properties of this pair depend strongly on the choice of translations.
Let $\bold{a}^k:=(a_1^k,a_2^k)
\in X^2\setminus\{0\}$ and $\bold{a}^k\to0$.
Then
$\Omega_1-a_1^k=\{(u,v)\in{\R^2}\mid v\ge-a_{12}^k\}$, $\Omega_2-a_2^k=\{(u,v)\in{\R^2}\mid v\le-a_{22}^k\}$ and $A^k:=(\Omega_1-a_1^k)\cap(\Omega_2-a_2^k) =\{(u,v)\in{\R^2}\mid -a_{12}^k\le v\le-a_{22}^k\}$.
Thus, $\zeta(\bold{a}^k)=+\infty$
if $a_{12}^k<a_{22}^k$, and $\zeta(\bold{a}^k)=\max\{-a_{12}^k,a_{22}^k,0\}\le\|\bold{a}^k\|$ otherwise.
If $a_{12}^k<a_{22}^k$ for all sufficiently large $k\in\N$, then $\rho(\bx;\siga)=+\infty$. Otherwise, $\rho(\bx;\siga)=\liminf_{k\to+\infty}\zeta(\bold{a}^k)=0$ and $\al(\bx;\siga)=\limsup_{k\to+\infty} \|\bold{a}^k\|/\zeta(\bold{a}^k){\ge1}$.
\sloppy

In the first case, $\{\Omega_1,\Omega_2\}$ is globally extremal at $\bx$ with respect to $\siga$.
A typical example is given by $a_1^k:=(0,0)$, $a_2^k:=(0,1/k)$.
This is basically the conventional extremality as in Definition~\ref{D1.1}\eqref{D1.1.1}.

In the second case, $\{\Omega_1,\Omega_2\}$ is neither extremal nor stationary at $\bx$ with respect to $\siga$.
Moreover, $\{\Omega_1,\Omega_2\}$ is not approximately stationary at $\bx$ with respect to $\siga$.
Indeed, if $\Omega_1\ni u^k_1\to0$, $\Omega_2\ni u^k_2\to0$ and $\bold{u}^k:=(u_1^k,u_2^k)$, then $u_{12}^k\ge0$ and $u_{22}^k\le0$; hence, $u_{12}^k+a_{12}^k\ge a_{12}^k$ and $u_{22}^k+a_{22}^k\le a_{22}^k$.
Similarly to the above, $$\zeta(\bold{u}^k+\bold{a}^k) =\max\{-u_{12}^k-a_{12}^k,u_{22}^k+a_{22}^k,0\} \le \zeta(\bold{a}^k)\le\|\bold{a}^k\|,$$
and consequently, $\al(\sigu;\siga)\ge1$, where $\sigu:=\{(u^k_1,u^k_2)\}$.
Thus, $\widetilde\al(\bx;\siga)\ge1$.
Hence, this case is of little interest from the point of view of applications.

The pair $\{\Omega_1,\Omega_2\}$ can be $\al$-stationary at $\bx$ with respect to $\siga$ with some $\al>1$.
For instance, if $a_1^k=a_2^k=(0,1/k)$, then $\|\bold{a}^k\|=\zeta(\bold{a}^k)=1/k$, and consequently, $\al(\bx;\siga)=1$, i.e., $\{\Omega_1,\Omega_2\}$ is $\al$-sta\-tionary at $\bx$ with respect to $\siga$ with any $\al>1$.
On the other hand, if $a_1^k=(0,1/k)$ and $a_2^k=(0,-1/k)$, then $\zeta(\bold{a}^k)=0$ and $\al(\bx;\siga)=+\infty$.
Alternatively, one can take
$a_1^k=(0,1/k^{1/2})$ or $a_1^k=(1/k^{1/2},1/k)$, and $a_2^k=(0,1/k)$.
Then $\|\bold{a}^k\|=1/k^{1/2}$, $\zeta(\bold{a}^k)=1/k$, and consequently, $\al(\bx;\siga)=+\infty$.
Thus, in these situations $\{\Omega_1,\Omega_2\}$ cannot be $\al$-stationary at $\bx$ with respect to $\siga$ with whatever $\al>0$.

Note also that the first components of the vectors $a_1^k$ and $a_2^k$ play little or no role in this example.
This is, of course, not surprising having in mind the geometry of the sets.
\end{example}

The next two examples illustrate differences between the properties in Definition~\ref{D2.01}; see Fig.~\ref{fig}.

\begin{example}
\label{E2.6}
Let $X:=\R^2$,
$\Omega_1:=\{(u,v)\in X\mid u^2+v\ge0\}$, $\Omega_2:=\{(u,v)\in X\mid v\le0\}$ and $\bx:=(0,0)$.
Take $a_1^k:=(0,0)$, $a_2^k:=(0,1/k)$ and $\bold{a}^k:=(a_1^k,a_2^k)$.
Then $\|\bold{a}^k\|=1/k$, $\Omega_1-a_1^k=\Omega_1$, $\Omega_2-a_2^k=\{(u,v)\in{\R^2}\mid v\le-1/k\}$, $A^k:=(\Omega_1-a_1^k)\cap(\Omega_2-a_2^k) =\{(u,v)\in{\R^2}\mid -u^2\le v\le-1/k\}$ and
$\zeta(\bold{a}^k)=1/k^{1/2}$.
Hence, $\rho(\bx;\siga)=\al(\bx;\siga)=0$, and consequently, $\{\Omega_1,\Omega_2\}$ is stationary but not extremal at $\bx$ with respect to $\siga$.
\end{example}

\begin{example}
\label{E2.7}
Let $X:=\R^2$, $\Omega_1:=\{(u,v)\in X\mid v\ge0\;\;\text{or}\;\;u+v\ge0\}$, $\Omega_2:=\{(u,v)\in X\mid v\le0\}$,
 and $\bx:=(0,0)$.
Take $a_1^k:=(0,0)$, $a_2^k:=(0,1/k^2)$ and $\bold{a}^k:=(a_1^k,a_2^k)$.
Then $\|\bold{a}^k\|=1/k^2$, $\Omega_1-a_1^k=\Omega_1$, $\Omega_2-a_2^k=\{(u,v)\in{\R^2}\mid v\le-1/k^2\}$, $A^k:=(\Omega_1-a_1^k)\cap(\Omega_2-a_2^k) =\{(u,v)\in{\R^2}\mid -u\le v\le-1/k^2\}$ and
$\zeta(\bold{a}^k)=1/k^2$.
Hence, $\al(\bx;\siga)=1$, and consequently, $\{\Omega_1,\Omega_2\}$ is not $\al$-stationary at $\bx$ with respect to $\siga$ with any $\al\in(0,1]$.

Now, take $x_{12}=x_{21}=x_{22}=0$, $x_{11}=-1/k$ and $\bold{x}^k:=(x_1^k,x_2^k)$.
Then $\Omega_1-x_1^k-a_1^k=\{(u,v)\in\R^2\mid v\ge0\;\;\text{or}\;\;u+v\ge1/k\}$, $\Omega_2-x_2^k-a_2^k=\{(u,v)\in{\R^2}\mid v\le-1/k^2\}$, $B^k:=(\Omega_1-x_1^k-a_1^k)\cap(\Omega_2-x_2^k-a_2^k) =\{(u,v)\in{\R^2}\mid 1/k-u\le v\le-1/k^2\}$ and
$\zeta(\bold{x}^k+\bold{a}^k)=1/k+1/k^2$.
Hence, $\al(\sigx;\siga) =\lim_{k\to+\infty}\frac{1/k^2}{1/k+1/k^2}=0$, and consequently, $\{\Omega_1,\Omega_2\}$ is stationary at $\sigx$ with respect to $\siga$.
By Definition~\ref{D2.01}\eqref{D2.01.3}, $\widetilde\al(\bx;\siga)=0$, and consequently, $\{\Omega_1,\Omega_2\}$ is approximately stationary at $\bx$ with respect to $\siga$.
\end{example}

\begin{figure}[H]
\centering
\includegraphics[width=.8\linewidth]{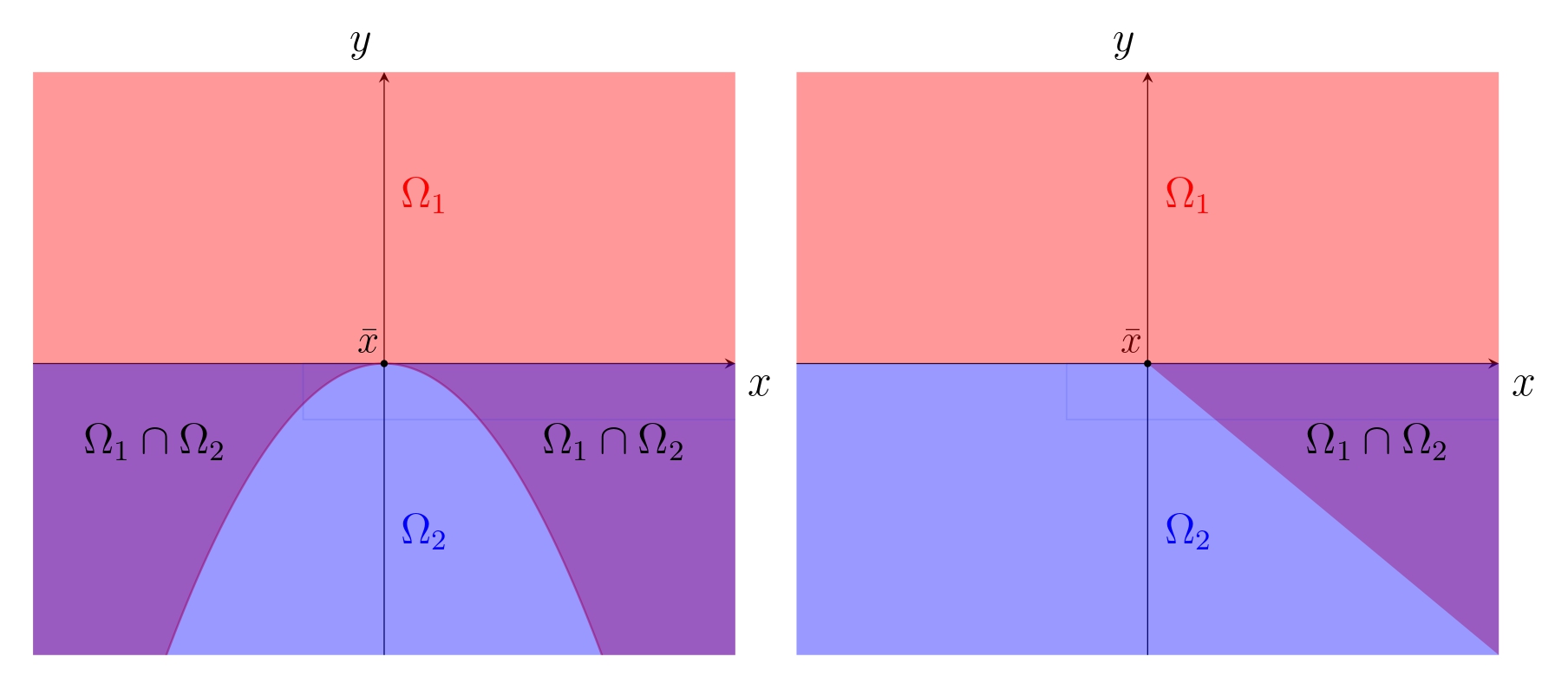}
\\
{\scriptsize
\hspace{15mm}Stationarity \hspace{35mm}
Approximate stationarity}\hspace{10mm}
\caption{Examples~\ref{E2.6} and \ref{E2.7}}
\label{fig}
\end{figure}

The sequential extremality and approximate stationarity properties possess certain stability.

\begin{proposition}
\label{P2.6}
Let
$\bold{u}^k\in\widehat\Omega$ $(k\in\N)$, $\bold{u}^k-\bold{x}^k\to0$ as $k\to+\infty$ and $\sigma_{\bold{u}}:=\{\bold{u}^k\}$.
Then
$\rho(\sigx;\siga) =\rho(\sigma_{\bold{u}};\sigma_{\bold{a}'})$ where $\sigma_{\bold{a}'}:=\{\bold{a}'^k\}$, $\bold{a}'^k:=\bold{a}^k+\bold{x}^k-\bold{u}^k$ $(k\in\N)$, and $\widetilde\al(\sigx;\siga) =\widetilde\al(\sigma_{\bold{u}};\siga)$.
As a consequence,
\Esa\ \iff \EE{E}{\sigma_{\bold{u}};\sigma_{\bold{a}'}}{}
and
\ASsa\ \iff \EE{AS}{\sigma_{\bold{u}};\siga}{}.

In particular, if $\bx\in\bigcap_{i=1}^n\Omega_i$, and $x^k_i\to\bx$ as $k\to+\infty$ $(i=1,\ldots,n)$, then\\
\Epa\ \iff \EE{E}{\sigx;\sigma_{\bold{a}'}}{} where $\bold{a}'^k:=\bold{a}^k+\bx-\bold{x}^k$ $(k\in\N)$,
and
\ASpa\ \iff \EE{AS}{\sigx;\siga}{}.
\end{proposition}

\begin{remark}
\begin{enumerate}
\item
Under the assumptions of Proposition~\ref{P2.6}, if $\rho(\sigx;\siga)>0$, then $\bold{a}'^k\ne0$ for all sufficiently large $k\in\N$.
\item
The formulations of stability of the extremality and approximate stationarity properties in Proposition~\ref{P2.6} are different.
The approximate stationarity equality $\widetilde\al(\sigx;\siga) =\widetilde\al(\sigma_{\bold{u}};\siga)$ involves different sequences $\sigx$ and $\sigu$ (such that $\bold{u}^k-\bold{x}^k\to0$) and holds with the same sequence $\siga$, while to ensure the extremality equality $\rho(\sigx;\siga) =\rho(\sigma_{\bold{u}};\sigma_{\bold{a}'})$, we need a different sequence $\sigma_{\bold{a}'}$ (satisfying $\bold{a}'^k-\bold{a}^k\to0$).
\item
In view of Example~\ref{E2.7}, the (non-approximate) stationarity lacks stability.
\end{enumerate}
\end{remark}

The next proposition gives characterisations of the properties in Definition~\ref{D2.01} holding on some subsequence of the given sequence $\{\bold{x}^k,\bold{a}^k\}$.

\begin{proposition}
\label{P2.9}
Let $\{\bold{x}'^k,\bold{a}'^k\}$ be a subsequence of $\{\bold{x}^k,\bold{a}^k\}$, $\sigma_{\bold{x}'}:=\{\bold{x}'^k\}$ and $\sigma_{\bold{a}'}:=\{\bold{a}'^k\}$.
Then
\begin{enumerate}
\item
\label{P2.9.1}
\EE{E}{\sigma_{\bold{x}'};\sigma_{\bold{a}'}}{}\ \iff
there is a $\rho>0$ such that, for any $\varepsilon>0$, there exists an integer $k>\eps\iv$ such that $\|\bold{a}^k\|<\eps$ and
$\zeta(\bold{x}^k+\bold{a}^k)>\rho$;

\item
\EE{GE}{\sigma_{\bold{x}'};\sigma_{\bold{a}'}}{}\ \iff
the above condition holds for any $\rho>0$;

\item
\EE{S}{\sigma_{\bold{x}'};\sigma_{\bold{a}'}}{\al}\ \iff
there is an $\al'\in(0,\al)$ such that,
for any $\varepsilon>0$, there exists an integer $k>\eps\iv$ such that $\|\bold{a}^k\|<\eps$ and $\|\bold{a}^k\|<\al'\zeta(\bold{x}^k+\bold{a}^k)$;
\item
\EE{S}{\sigma_{\bold{x}'};\sigma_{\bold{a}'}}{}\ \iff
for any $\varepsilon>0$, there exists an integer $k>\eps\iv$ such that $\|\bold{a}^k\|<\eps$ and $\|\bold{a}^k\|<\eps\zeta(\bold{x}^k+\bold{a}^k)$;
\item
\EE{AS}{\sigma_{\bold{x}'};\sigma_{\bold{a}'}}{\al}\ \iff
there is an $\al'\in(0,\al)$ such that,
for any $\varepsilon>0$, there exist an integer $k>\eps\iv$ and points $x_i'\in\Omega_i\cap B_\eps(x^k_i)$ $(i=1,\ldots,n)$ such that $\|\bold{a}^k\|<\eps$ and $\|\bold{a}^k\|<\al' \zeta(\bold{x}'+\bold{a}^k)$, where  $\bold{x}':=(x_1',\ldots,x_n')$;
\item
\EE{AS}{\sigma_{\bold{x}'};\sigma_{\bold{a}'}}{}\ \iff
for any $\varepsilon>0$, there exist an integer $k>\eps\iv$ and points $x_i'\in\Omega_i\cap B_\eps(x^k_i)$ $(i=1,\ldots,n)$ such that $\|\bold{a}^k\|<\eps$ and $\|\bold{a}^k\|<\eps \zeta(\bold{x}'+\bold{a}^k)$, where  $\bold{x}':=(x_1',\ldots,x_n')$.
\end{enumerate}
\end{proposition}

All the assertions in Proposition~\ref{P2.9} can be proved using basically the same arguments.
Below, we provide a proof of assertion \eqref{P2.9.1} as an illustration.

\begin{proof}
[of Proposition~\ref{P2.9}\eqref{P2.9.1}]
Let $\N\ni k^j\to+\infty$ as $j\to+\infty$,
$\{\Omega_1,\ldots,\Omega_n\}$ be extremal at $\sigma_{\bold{x}'}:=\{\bold{x}^{k^j}\}$ with respect to $\sigma_{\bold{a}'}:=\{\bold{a}^{k^j}\}$, and let $\eps>0$.
By Definition~\ref{D2.01}\eqref{D2.01.1}, $\rho({\sigma_{\bold{x}'};\sigma_{\bold{a}'}}{})>0$.
Take any $\rho\in\big(0,\rho({\sigma_{\bold{x}'};\sigma_{\bold{a}'}}{})\big)$.
By \eqref{D2.01-1}, there is a $j_0\in\N$ such that $\zeta(\bold{x}^{k^j}+\bold{a}^{k^j})>\rho$ for all integers $j>j_0$.
Since $\bold{a}^k\to0$, we can choose an integer $j>j_0$ so that $k:=k^j>\eps\iv$ and $\|\bold{a}^k\|<\eps$.
Condition $\zeta(\bold{x}^k+\bold{a}^k)>\rho$ is satisfied automatically.

Conversely, suppose that there is a $\rho>0$ such that, for any $\varepsilon>0$, there exists an integer $k>\eps\iv$ such that $\|\bold{a}^k\|<\eps$ and $\zeta(\bold{x}^k+\bold{a}^k)>\rho$.
Let $j\in\N$.
By the assumption, there exists an integer $k^j>j$ such that $\|\bold{a}^{k^j}\|<1/j$ and $\zeta(\bold{x}^{k^j}+\bold{a}^{k^j})>\rho$.
Denote $\sigma_{\bold{x}'}:=\{\bold{x}^{k^j}\}$ and $\sigma_{\bold{a}'}:=\{\bold{a}^{k^j}\}$.
Hence, $\rho({\sigma_{\bold{x}'};\sigma_{\bold{a}'}}{})\ge\rho>0$.
By Definition~\ref{D2.01}\eqref{D2.01.1}, $\{\Omega_1,\ldots,\Omega_n\}$ is extremal at $\sigma_{\bold{x}'}$ with respect to $\sigma_{\bold{a}'}$.
\qed\end{proof}

\begin{corollary}
Let $\bx\in\bigcap_{i=1}^n\Omega_i$, and
$\sigaprime$ be a subsequence of $\siga$.
Then
\begin{enumerate}
\item
\EE{E}{\bx;\sigma_{\bold{a}'}}{}\ \iff
there is a $\rho>0$ such that, for any $\varepsilon>0$, there exists an integer $k>\eps\iv$ such that $\|\bold{a}^k\|<\eps$ and
$d\big(\bx,\bigcap_{i=1}^n(\Omega_i-a^k_i)\big)>\rho$;

\item
\EE{GE}{\bx;\sigma_{\bold{a}'}}{}\ \iff
the above condition holds for any $\rho>0$;

\item
\EE{S}{\bx;\sigma_{\bold{a}'}}{\al}\ \iff
there is an $\al'\in(0,\al)$ such that,
for any $\varepsilon>0$, there exists an integer $k>\eps\iv$ such that $\|\bold{a}^k\|<\eps$ and $\|\bold{a}^k\|<\al' d\big(\bx,\bigcap_{i=1}^n(\Omega_i-a^k_i)\big)$;

\item
\EE{S}{\bx;\sigma_{\bold{a}'}}{}\ \iff
for any $\varepsilon>0$, there exists an integer $k>\eps\iv$ such that $\|\bold{a}^k\|<\eps$ and $\|\bold{a}^k\|<\eps d\big(\bx,\bigcap_{i=1}^n(\Omega_i-a^k_i)\big)$;
\item
\EE{AS}{\bx;\sigma_{\bold{a}'}}{\al}\ \iff
for any $\varepsilon>0$, there exist an integer $k>\eps\iv$ and points $x_i'\in\Omega_i\cap B_\eps(\bx)$ $(i=1,\ldots,n)$ such that $\|\bold{a}^k\|<\eps$ and $\|\bold{a}^k\|<\al' \zeta(\bold{x}'+\bold{a}^k)$, where  $\bold{x}':=(x_1',\ldots,x_n')$;
\item
\EE{AS}{\bx;\sigma_{\bold{a}'}}{}\ \iff
for any $\varepsilon>0$, there exist an integer $k>\eps\iv$ and points $x_i'\in\Omega_i\cap B_\eps(\bx)$ $(i=1,\ldots,n)$ such that $\|\bold{a}^k\|<\eps$ and $\|\bold{a}^k\|<\eps \zeta(\bold{x}'+\bold{a}^k)$, where  $\bold{x}':=(x_1',\ldots,x_n')$.
\end{enumerate}
\end{corollary}

\subsection{Sequential extremality, stationarity and approximate stationarity}
\label{S2.3}

In this subsection, we move closer to the sequential definitions studied in \cite{CuoKru} as well as the conventional concepts of extremality and stationarity
of a collection of sets
and consider a more traditional setting when the extremality and stationarity properties of $\{\Omega_1,\ldots,\Omega_n\}$ at a fixed sequence $\sigx$ hold for \emph{some} sequence $\siga$.
The definitions and statements below rely on the corresponding ones in Section~\ref{S2.01}.

\begin{definition}
\label{D2.11}
The collection $\{\Omega_1,\ldots,\Omega_n\}$ is
extremal (resp., globally extremal, $\al$-sta\-tio\-nary, stationary, approximately $\al$-stationary, approximately stationary) at $\sigx$ if it is extremal (resp., globally extremal, $\al$-sta\-tionary, stationary, approximately $\al$-stationary, approximately stationary) at $\sigx$ with respect to some $\siga:=\{\bold{a}^k\}\subset X^n\setminus\{0\}$ $(k\in\N)$ such that $\bold{a}^k\to0$ as $k\to+\infty$.
The corresponding properties are denoted by \Es, \GEs, \Ssxal, \Ss, \ASsxal\ and \ASs, respectively.

If $x_i^k=\bx$ for some $\bx\in\bigcap_{i=1}^n\Omega_i$ and all $i=1,\ldots,n$ and $k=1,2,\ldots$,
we say that $\{\Omega_1,\ldots,\Omega_n\}$ is extremal (resp., globally extremal, $\al$-sta\-tionary, stationary, approximately $\al$-sta\-tionary, approximately stationary) at $\bx$.
The corresponding properties are denoted by \Ep, \GEp, \Spxal, \Sp, \ASpxal\ or \ASp.
\end{definition}

The next three statements are consequences of
the corresponding ones in Section~\ref{S2.01}.

\begin{proposition}
\label{P2.12}
\GEs\ $\Rightarrow$ \Es\ $\Rightarrow$ \Ss\ $\Rightarrow$ \ASs, and \Ssxal\ $\Rightarrow$ \ASsxal\; for any ${\al>0}$.
In particular, if $\bx\in\bigcap_{i=1}^n\Omega_i$, then\\
\GEp\ \folgt \Ep\ \folgt \Sp\ \folgt \ASp, and
\Spxal~\folgt \ASpxal\; for any $\al>0$.
\end{proposition}

\begin{proposition}
Let {\rm\textbf{P}} stand for any of the properties: {\rm\textbf{E}, \textbf{GE}, \textbf{S$_\al$}, \textbf{S}, \textbf{AS$_\al$} or \textbf{AS}}.\\
Then \EE{P}{\sigx}{} \folgt \EE{P}{\sigxprime}{}\; for any subsequence $\sigxprime$ of $\sigx$.
\end{proposition}

\begin{proposition}
\label{P2.14}
Let $\sigu:=\{\bold{u}^k\}\subset\widehat\Omega$ and $\bold{u}^k-\bold{x}^k\to0$ as $k\to+\infty$.
Then \EE{P}{\sigx}{} \iff \EE{P}{\sigu}{}, where {\rm\textbf{P}} stands for any of the properties: {\rm\textbf{E}, \textbf{GE}, \textbf{AS$_\al$} or \textbf{AS}}.
In particular, if $x^k_i\to\bx$ as $k\to+\infty$ for some $\bx\in\bigcap_{i=1}^n\Omega_i$ and all $i=1,\ldots,n$, then \EE{P}{\bx}{} \iff \EE{P}{\sigx}{}.
\end{proposition}

In view of Proposition~\ref{P2.14}, the sequential extremality and approximate stationarity properties in the sense of Definition~\ref{D2.11} possess certain stability.
The next modification of Example~\ref{E2.7} shows that the
(non-approximate) stationarity lacks stability.

\begin{example}
\label{E2.16}
Let $X:=\R^2$, $\Omega_1:=\{(u,v)\in X\mid v\ge0\;\;\text{or}\;\;u+v\ge0\}$, $\Omega_2:=\{(u,v)\in X\mid v\le0\}$, and $\bx:=(0,0)$.
It was shown in Example~\ref{E2.7} that $\{\Omega_1,\Omega_2\}$ is stationary at some $\sigx$ converging to $(\bx,\bx)$ with respect to some $\siga$ converging to $0$, hence, just stationary at $\sigx$, while it is not stationary at $\bx$ with respect to that $\siga$.
It remains to check that $\{\Omega_1,\Omega_2\}$ is not stationary at $\bx$ with respect to any $\siga$ converging to $0$.

Let $a_1^k:=(a_{11}^k,a_{12}^k)$, $a_2^k:=(a_{21}^k,a_{22}^k)$, $\bold{a}^k:=(a_1^k,a_2^k)$, and $\bold{a}^k\to0$.
Then $\Omega_1-a_1^k=\{(u,v)\in\R^2\mid v\ge-a_{12}^k \;\;\text{or}\;\;u+v\ge-a_{11}^k-a_{12}^k\}$, $\Omega_2-a_2^k=\{(u,v)\in{\R^2}\mid v\le-a_{22}^k\}$ and $A^k:=(\Omega_1-a_1^k)\cap(\Omega_2-a_2^k) =\{(u,v)\in{\R^2}\mid -\max\{u,a_{11}^k\}-a_{12}^k\le v\le-a_{22}^k\}$.
If $a_{11}^k\ge a_{22}^k-a_{12}^k$, then $\zeta(\bold{a}^k)=(a_{22}^k)_+$.
Otherwise, $\zeta(\bold{a}^k)=\max\{a_{22}^k,a_{22}^k-a_{12}^k,0\}$.
Thus, $\zeta(\bold{a}^k)\le2\|\bold{a}^k\|$.
Hence, $\al(\bx;\siga)\ge1/2$, and consequently, $\{\Omega_1,\Omega_2\}$ is not stationary at $\bx$ with respect to $\siga$.
\end{example}

\if{
\NDC{17/5/26.
Let $\Omega_1:=\{(x,y)\in\R^2\mid y\le|x|\}$ and  $\Omega_2:=\{(x,y)\in\R^2\mid y\ge-|x|\}$, and $\bx:=(0,0)$.
Then one can check that $\Omega_1$ and $\Omega_2$ is not stationary at $\bx$.
Let $x^k_1:=(1/k,1/k)\in\Omega_1$, $x^k_2:=(1/k,-1/k)\in\Omega_2$, $a^k_1:=(0,1/k^2)$, $a^k_2:=(0,-1/k^2)$.
Then $A_k:=(\Omega_1-x^k_1-a^k_1)\cap (\Omega_2-x^k_2-a^k_2)=\left\{(u,v\in\R^2)\mid -\abs{u+\frac{1}{k}}+\frac{1}{k}+\frac{1}{k^2}\le v\le\abs{u+\frac{1}{k}}-\frac{1}{k}-\frac{1}{k^2}\right\}.
$
Then $d(0,A_k)=\inf_{(u,v)\in A_k}\max\{|u|,|v|\}$.
We have $\abs{u+\frac{1}{k}}=\frac{1}{k}+\frac{1}{k^2}$, or equivalently, $u=\frac{1}{k^2}$ or $u=-\frac{2}{k}-\frac{1}{k^2}$.
In these case, $v=0$.
Thus, $d(0,A_k)=\frac{1}{k^2}$.
We have $\|\bold{a}^k\|=\frac{1}{k^2}$.
Thus, $\al(\sigx;\siga)=1$.
}
}\fi

\begin{proposition}
\label{P2.15}
Let $\sigxprime$ be a subsequence of $\sigx$.
Then
\begin{enumerate}
\item
\EE{E}{\sigxprime}{}\ \iff
there is a $\rho>0$ such that, for any $\varepsilon>0$, there exist an integer $k>\eps\iv$ and points $a_1,\ldots,a_n\in\eps\B_{X}$ such that
$\zeta(\bold{x}^k+\bold{a})>\rho$, where $\bold{a}:=(a_1,\ldots,a_n)$;

\item
\label{P2.15.2}
\EE{GE}{\sigxprime}{}\ \iff
the above condition holds for any $\rho>0$;

\item
\EE{S}{\sigxprime}{\al}\ \iff
there is an $\al'\in(0,\al)$ such that, for any $\varepsilon>0$, there exist an integer $k>\eps\iv$ and points $a_1,\ldots,a_n\in{\eps}\B_{X}$ such that $\|\bold{a}\|<\al'\zeta(\bold{x}^k+\bold{a})$, where $\bold{a}:=(a_1,\ldots,a_n)$;

\item
\EE{S}{\sigxprime}{}\ \iff
for any $\varepsilon>0$, there exist an integer $k>\eps\iv$ and points $a_1,\ldots,a_n\in{\eps}\B_{X}$ such that $\|\bold{a}\|<\eps\zeta(\bold{x}^k+\bold{a})$, where $\bold{a}:=(a_1,\ldots,a_n)$;

\item
\EE{AS}{\sigxprime}{\al}\ \iff
there is an $\al'\in(0,\al)$ such that, for any $\varepsilon>0$, there exist an integer $k>\eps\iv$, and points $x_i'\in\Omega_i\cap B_\eps(x^k_i)$ and $a_i\in{\eps}\B_{X}$ $(i=1,\ldots,n)$ such that $\|\bold{a}\|<\al'\zeta(\bold{x}'+\bold{a})$, where $\bold{x}':=(x'_1,\ldots,x'_n)$ and $\bold{a}:=(a_1,\ldots,a_n)$;

\item
\EE{AS}{\sigxprime}{}\ \iff
for any $\varepsilon>0$, there exist an integer $k>\eps\iv$, and points $x_i'\in\Omega_i\cap B_\eps(x^k_i)$ and $a_i\in{\eps}\B_{X}$ $(i=1,\ldots,n)$ such that $\|\bold{a}\|<\eps\zeta(\bold{x}'+\bold{a})$, where $\bold{x}':=(x'_1,\ldots,x'_n)$ and $\bold{a}:=(a_1,\ldots,a_n)$.
\end{enumerate}
\end{proposition}


The proofs of all the assertions in the above proposition use standard arguments already demonstrated in the proof of Proposition~\ref{P2.9}\eqref{P2.9.1}.
If $x_i^k=\bx$ for some $\bx\in\bigcap_{i=1}^n\Omega_i$ and all $i=1,\ldots,n$ and $k=1,2,\ldots$, the statement takes a simpler form and does not involve subsequences.

\begin{corollary}
\label{C2.16}
Let $\bx\in\bigcap_{i=1}^n\Omega_i$.
\begin{enumerate}
\item
\Ep\ \iff
there is a $\rho>0$ such that,
for any $\eps>0$, there exist $a_1,\ldots,a_n\in\eps\B_{X}$ such that
$d\big(\bx,\bigcap_{i=1}^n(\Omega_i-a_i)\big)>\rho$;

\item
\GEp\ \iff
the above condition holds for any $\rho>0$;

\item
\Spxal\ \iff
there is an $\al'\in(0,\al)$ such that,
for any $\eps>0$, there exist points $a_1,\ldots,a_n\in{\eps}\B_X$ such that $\max_{1\le i\le n}\|a_i\|<\al' d\big(\bx,\bigcap_{i=1}^n(\Omega_i-a_i)\big)$;

\item
\Sp\ \iff
for any $\varepsilon>0$, there exist points $a_1,\ldots,a_n\in{\eps}\B_X$ such that $\max_{1\le i\le n}\|a_i\|<\eps d\big(\bx,\bigcap_{i=1}^n(\Omega_i-a_i)\big)$;

\item
\ASpxal\ \iff
there is an $\al'\in(0,\al)$ such that, for any $\varepsilon>0$, there exist points $x_i'\in\Omega_i\cap B_\eps(\bx)$ and $a_i\in{\eps}\B_{X}$ $(i=1,\ldots,n)$ such that $\|\bold{a}\|<\al'\zeta(\bold{x}'+\bold{a})$, where $\bold{x}':=(x'_1,\ldots,x'_n)$ and $\bold{a}:=(a_1,\ldots,a_n)$;

\item
\ASp\ \iff
for any $\varepsilon>0$, there exist points $x_i'\in\Omega_i\cap B_\eps(\bx)$ and $a_i\in{\eps}\B_{X}$ $(i=1,\ldots,n)$ such that $\|\bold{a}\|<\eps\zeta(\bold{x}'+\bold{a})$, where $\bold{x}':=(x'_1,\ldots,x'_n)$ and $\bold{a}:=(a_1,\ldots,a_n)$.
\end{enumerate}
\end{corollary}

\begin{remark}
In view of Proposition~\ref{P2.15}, in the particular case when $x_i^1=x_i^2=\ldots$ $(i=1,\ldots,n)$, properties \Es, \Ss\ and \ASs\ reduce to the corresponding `relative' properties in \cite[Definition~3]{BuiKru18} (studied there for $n=2$).
If all these points coincide (the case covered by Corollary~\ref{C2.16}),
properties \Ep, \Sp\ and \ASp\ are equivalent to the corresponding extremality, stationarity and approximate stationarity properties in Definition~\ref{D1.1}.
\end{remark}

In the convex case, the statement of Proposition~\ref{P2.12} can be partially reversed.

\begin{proposition}
\label{P2.16}
Let $\Omega_1,\ldots,\Omega_n$ be convex.
Then
\GEs\ \iff \Es\ \iff \Ss\ \iff \ASs.

In particular, if $\bx\in\bigcap_{i=1}^n\Omega_i$, then
\GEp\ \iff \Ep\ \iff \Sp\ \iff \ASp.
\end{proposition}

\begin{proof}
Thanks to Proposition~\ref{P2.12}, we only need to show that \ASs\ \folgt \GEs.
Suppose that $\{\Omega_1,\ldots,\Omega_n\}$ is not globally extremal at $\sigx$.
By Proposition~\ref{P2.15}\eqref{P2.15.2}, there exist numbers $\rho>0$ and $\eps>0$ such that
\begin{gather}
\label{P2.16P1}
d\Big(0,\bigcap_{i=1}^n(\Omega_i-x_i^k-a_i)\Big)\le\rho
\end{gather}
for all integers $k>\eps\iv$ and all $a_1,\ldots,a_n\in\eps\B_X$.
Fix a $\de\in(0,\min\{\rho,\varepsilon/(1+\rho)\})$.

Let $\bold{u}^k:=(u^k_1,\ldots,u^k_n)\in\widehat\Omega$, $\bold{a}^k:=(a^k_1,\ldots,a^k_n)\in X^n\setminus\{0\}$, $\bold{u}^k-\bold{x}^k\to0$ and $\bold{a}^k\to0$ as $k\to+\infty$.
Choose a $k_0>\de\iv$ $(>\eps\iv)$ so that $\|\bold{u}^k-\bold{x}^k\|<\de$ and $\|\bold{a}^k\|<\de^2$ for all integers $k>k_0$.
Take any integer $k>k_0$, and set $t:=\|\bold{a}^k\|/(\de\rho)$.
Then $0<t<\de/\rho<1$ and
$\|\bold{u}^k-\bold{x}^k+\bold{a}^k/t\|<\de(1+\rho) {<\eps}$.
By \eqref{P2.16P1}, $\zeta(\bold{u}^k+\bold{a}^k/t)\le\rho$, and consequently,
there is an $x\in\rho\overline\B_X$ such that
$x+u_i^k+a_i^k/t\in\Omega_i$ $(i=1,\ldots,n)$.
Since the sets are convex, we have $t(x+u_i^k+a_i^k/t)+(1-t)u_i^k\in\Omega_i$, or equivalently, $tx+u_i^k+a_i^k\in\Omega_i$ $(i=1,\ldots,n)$, i.e., $tx\in\bigcap_{i=1}^n\big(\Omega_i-u_i^k-a_i^k\big)$.
Hence,
$\zeta(\bold{u}^k+\bold{a}^k)\le t\|x\|\le t\rho=\|\bold{a}^k\|/\de$, and consequently, $\|\bold{a}^k\|/\zeta(\bold{u}^k+\bold{a}^k)\ge\de$
for all $k>k_0$.
By Definitions~\ref{D2.01} and \ref{D2.11}, $\{\Omega_1,\ldots,\Omega_n\}$ is not approximately stationary at~$\sigx$.
\qed\end{proof}

\begin{remark}
\label{R2.16}
As discussed in the Introduction,
the sequential Definition~\ref{D2.11} (together with Definition~\ref{D2.01} which it builds upon) partially extend and improve \cite[Definitions~2.1 and~3.1]{CuoKru} along several lines.
\begin{enumerate}
\item
The model adopted in the current paper does not assume condition \eqref{D1.3-1} as it is superfluous when proving dual necessary conditions.
\item
Ignoring the superfluous condition \eqref{D1.3-1}, thanks to Proposition~\ref{P2.15}, extremality and stationarity at $\{x_i^k\}$ $(i=1,\ldots,n)$
in the sense of Definition~\ref{D1.3} are equivalent to the corresponding properties in the sense of Definition~\ref{D2.11} at some subsequence of $\sigx$.
Thus, the definitions adopted in this paper
are more specific in identifying the sequences relevant for the properties.
\item
The definition of approximate stationarity in part \eqref{D2.01.3} of Definition~\ref{D2.01} differs from that in Definition~\ref{D1.3}\eqref{D1.3.3}.
Nevertheless, they are largely equivalent up to taking subsequences (ignoring the superfluous condition \eqref{D1.3-1}).
\end{enumerate}
In view of the above observations,
Propositions~\ref{P2.14} and \ref{P2.16} extend \cite[Proposition~2.4(i) and (iii), and Proposition~2.7]{CuoKru}, respectively.
Example~\ref{E2.16} shows, in particular, that the claim in \cite[Proposition~2.4(ii)]{CuoKru} is not entirely correct.
Many examples provided in \cite{CuoKru} are applicable for illustrating the properties in Definition~\ref{D2.11}.
\end{remark}

\section{Generalised separation}
\label{S3}

In this section, we establish dual necessary (generalised separation) conditions for the sequential extremality and stationarity properties discussed in Section~\ref{S2}.
The underlying space $X$ is assumed here to be Banach, and the sets $\Omega_1,\ldots,\Omega_n$ closed (unless stated otherwise).

The dual conditions are formulated in terms of normal cones to the sets involved.
Recall our standing convention that the normal cone generic notation $N$ stands for $N^C$ if $X$ is a general Banach space, and $N^F$ if $X$ is Asplund.
Thus, the definition of the \GSsal\ property below covers both cases.

In the statements, we employ certain generalised separation conditions.
Let $x^k_i\in\Omega_i$,
$x_i^{*k}\in N_{\Omega_i}(x_i^k)$
($i=1,\ldots,n$, $k=1,2,\ldots$),
and $\al>0$.

\begin{enumerate}[leftmargin=11mm]
\item
[\GSsal]
\emph{Generalised $\al$-separation}:
$\limsup_{k\to+\infty} \|{\sum_{i=1}^nx^{*k}_i}\|<\al$ and
$\sum_{i=1}^n\|x_i^{*k}\|=1$ $(k\in\N)$.
\end{enumerate}
If
$\lim_{k\to+\infty} \sum_{i=1}^nx^{*k}_i=0$, we talk about the \emph{generalised separation} and write \GSs.

The generalised separation conditions can be complemented by a certain auxiliary primal-dual condition connecting the dual variables (normals) in the generalised separation conditions and the corresponding primal variables.
Let
$\bold{z}^k\in X^n$, $\bold{x}^{*k}\in(X^*)^n$ and $\|\bold{x}^{*k}\|=1$ $(k\in\N)$.

\begin{enumerate}[leftmargin=9mm]
\item
[\PDs]
\emph{Complimentary primal-dual condition}:
$\bold{z}^k\ne0$ {for all
large} ${k\in\N}$, and
$\lim\limits_{k\to+\infty}\frac{\ang{\bold{x}^{*k},\bold{z}^k}} {\|\bold{z}^k\|}=1$.
\end{enumerate}

\begin{remark}
Condition \PDs\ roughly says that the primal unit vector $\bold{z}^k/\|\bold{z}^k\|$ and dual unit vector $\bold{x}^{*k}$ are getting almost `co-linear' as $k\to+\infty$.
More precisely, it is not difficult to show that, if $X$ is Hilbert, then $\big\|\frac{\bold{z}^k}{\|\bold{z}^k\|}-\bold{x}^{*k}\big\|\to0$ as $k\to+\infty$.
\end{remark}

\begin{theorem}
\label{T3.1}
Let $\al>0$.
Suppose that $\{\Omega_1,\ldots,\Omega_n\}$ is approximately $\al$-sta\-tionary (particularly, $\al$-sta\-tionary)
at $\sigxz:=\{\bold{x}^k_0\}\subset\widehat\Omega$.
Then there exist $\bold{x}^k:=(x^k_1,\ldots,x^k_n)$ and
$x_i^{*k}\in N_{\Omega_i}(x_i^k)$
$(i=1,\ldots,n,\; k\in\N)$ such that $\bold{x}^k-\bold{x}^k_0\to0$ as $k\to+\infty$, and condition \GSsal\ holds true.

Furthermore, if the assumed property holds with respect to some $\siga:=\{\bold{a}^k\}\subset X^n\setminus\{0\}$, $\bold{a}^k:=(a^k_1,\ldots,a^k_n)$ $(k\in\N)$ such that $\bold{a}^k\to0$ as $k\to+\infty$, then there exist $\bold{u}^k:=(u^k_1,\ldots,u^k_n)\in\widehat\Omega$ and $v^k\in X$ $(k\in\N)$
such that $\bold{u}^k-\bold{x}^k_0\to0$ ($\bold{u}^k=\bold{x}^k_0$ if $\{\Omega_1,\ldots,\Omega_n\}$ is $\al$-sta\-tionary at $\sigxz$ with respect to $\siga$), $v^k\to0$ as $k\to+\infty$, and condition \PDs\ is satisfied with $\bold{x}^{*k}:=(x^{*k}_1,\ldots,x^{*k}_n)$,
\begin{gather}
\label{T3.1-2}
z_i^k:=v^k-x_i^k+u_i^k+a_i^k\;\; (i=1,\ldots,n)\AND
\bold{z}^k:=(z^k_1,\ldots,z^k_n)\;\; (k\in\N).
\end{gather}
\end{theorem}

\begin{proof}
By Definition~\ref{D2.11}, $\{\Omega_1,\ldots,\Omega_n\}$ is approximately $\al$-sta\-tionary at $\sigxz$ with respect to some $\siga:=\{\bold{a}^k\}\subset X^n\setminus\{0\}$, $\bold{a}^k:=(a^k_1,\ldots,a^k_n)$ $(k\in\N)$ such that $\bold{a}^k\to0$ as $k\to+\infty$.
By Definition~\ref{D2.01}\eqref{D2.01.3}, there exist an $\al'\in(0,\al)$, a sequence $\bold{u}^k:=(u^k_1,\ldots,u^k_n)\in\widehat\Omega$ $(k\in\N)$, and a $k_0\in\N$ such that $\bold{u}^k-\bold{x}^k_0\to0$ as $k\to+\infty$ and, for all integers $k>k_0$, it holds $d(0,\bigcap_{i=1}^n\Omega'^k_i)>\rho^k$, where $\Omega'^k_i:=\Omega_i-u^k_i-a^k_i$ and $\rho^k:=\|\bold{a}^k\|/\al'$.
If $\{\Omega_1,\ldots,\Omega_n\}$ is $\al$-sta\-tionary at $\sigxz$ with respect to $\siga$, then, in view of Definition~\ref{D2.01}\eqref{D2.01.2},
$\bold{u}^k=\bold{x}^k_0$ $(k\in\N)$.
Hence,
\begin{gather}
\label{T3.1P0}
\bigcap_{i=1}^n \Omega'^k_i\cap(\rho^k\overline\B_X)=\emptyset.
\end{gather}
Denote $t^k:=\|\bold{a}^k\|$.
Taking a larger $k_0$ if necessary, we assume without loss of generality that $(1+t^k)(t^k)^{\frac12}<1$ for all integers $k>k_0$.

Fix an integer $k>k_0$.
Choose an $\al''\in(\al',\al)$.
Set
$\eps^k:=t^k+\min\{(t^k)^2,\rho^k(\al''-\al')\}$ and
$\de^k:=(t^k)^{\frac12}$.
Then $\eps^k>t^k$,
\begin{gather}
\label{T3.1P1}
\limsup_{k\to+\infty}\frac{\eps^k}{\rho^k} \le\al'+(\al''-\al')=\al''<\al
\AND \limsup_{k\to+\infty}\frac{\eps^k}{\de^k} \le\lim_{k\to+\infty}(1+t^k)(t^k)^{\frac12}=0.
\end{gather}
Note that $
-a^k_i\in\Omega'^k_i$ and $\|a^k_i\|<\eps^k$ $(i=1,\ldots,n)$.
By Theorem~\ref{T1.5}\eqref{T1.5-1}, there exist $x_i^k\in\Omega_i$, $x_i'^{*k}\in N^C_{\Omega_i}(x_i^k)$,
$y^{*k}_i\in X^*$ $(i=1,\ldots,n)$
and $v^k\in\rho^k\B_X$
such that
\begin{gather}
\label{T3.1P7}
\|x^k_i-u^k_i\|<\de^k\quad (i=1,\ldots,n),
\\
\label{T3.1P3}
\sum_{i=1}^n\|y_i^{*k}\|=1,
\\
\label{T3.1P2}
{\de^k} \sum_{i=1}^n\|x'^{*k}_i-y_i^{*k}\|+ \rho^k\Big\|\sum_{i=1}^ny_i^{*k}\Big\|<\eps^k,
\\
\label{T3.1P4}
\sum_{i=1}^n\ang{y_i^{*k},z_i^k}= \max_{1\le i\le n}\|z^k_i\|,
\end{gather}
where
$z_i^k
$ $(i=1,\ldots,n)$ are given by \eqref{T3.1-2}.
Moreover, $\bold{z}^k:=(z^k_1,\ldots,z^k_n)\ne0$ because otherwise $v^k\in\bigcap_{i=1}^n\Omega'^k_i\cap \big(\rho^k\B_X\big)$ contradicting \eqref{T3.1P0}.

Thus,
$x^k_i-x^k_{0i}\to0$ $(i=1,\ldots,n)$, and
$v^k\to0$
as $k\to+\infty$.
By \eqref{T3.1P2},
\begin{gather}
\label{T3.1P5}
\sum_{i=1}^n\|x'^{*k}_i-y_i^{*k}\| <\frac{\eps^k}{\de^k}<1.
\end{gather}
It follows from \eqref{T3.1P3} and \eqref{T3.1P5} that \begin{gather*}
1-\frac{\eps^k}{\de^k} <\sum_{i=1}^n\|x_i'^{*k}\| <1+\frac{\eps^k}{\de^k}
\end{gather*}
and, in view of \eqref{T3.1P1},
\begin{gather}\label{T3.1P8}
0\ne\eta^k:=\sum_{i=1}^n\|x_i'^{*k}\|\to1\;\; \text{as}\;\; k\to{+\infty}.
\end{gather}
Set $x_i^{*k}:=x_i'^{*k}/\eta^k$
$(i=1,\ldots,n)$ and $\bold{x}^{*k}:=(x^{*k}_1,\ldots,x^{*k}_n)$.
Thus, $x_i^{*k}\in N^C_{\Omega}(x_i^k)$ $(i=1,\ldots,n)$, $\|\bold{x}^{*k}\|=1$,
\begin{gather*}
\eta^k\Big\|\sum_{i=1}^nx_i^{*k}\Big\| \le\Big\|\sum_{i=1}^ny_i^{*k}\Big\| +\sum_{i=1}^n\|x_i'^{*k}-y_i^{*k}\| \overset{\eqref{T3.1P2},\eqref{T3.1P5}
}< \frac{\eps^k}{\rho^k} +
\frac{\eps^k}{\de^k},
\\
\eta^k\frac{\ang{\bold{x}^{*k},\bold{z}^k}} {\|\bold{z}^k\|} >\frac{\sum_{i=1}^n\ang{y_i^{*k},z_i^k}} {\|\bold{z}^k\|} -\sum_{i=1}^n\|x'^{*k}_i-y_i^{*k}\| \overset{\eqref{T3.1P4},\eqref{T3.1P5}}> 1-\frac{\eps^k}{\de^k}.
\end{gather*}
Taking limits in the above estimates, in view of \eqref{T3.1P8} and \eqref{T3.1P1},
 we obtain
\begin{gather}
\label{T3.1P9}
\limsup_{k\to+\infty} \Big\|{\sum_{i=1}^nx^{*k}_i}\Big\|<\al
\AND
\lim_{k\to+\infty}\frac{\ang{\bold{x}^{*k},\bold{z}^k}} {\|\bold{z}^k\|}=1.
\end{gather}

Let $X$ be Asplund.
By Theorem~\ref{T1.5}\eqref{T1.5-2},
there exist $x_i^k\in\Omega_i$, $x_i'^{*k}\in N^F_{\Omega_i}(x_i^k)$,
$y^{*k}_i\in X^*$ $(i=1,\ldots,n)$ and $v^k\in\rho^k\B_X$ such that conditions \eqref{T3.1P7}, \eqref{T3.1P3} and \eqref{T3.1P2} are satisfied, and
\begin{gather}
\label{T3.1P6}
\sum_{i=1}^n\ang{y_i^{*k},z_i^k}>\Big(1-\frac1k\Big) \max_{1\le i\le n}\|z^k_i\|,
\end{gather}
where
$z_i^k
$ $(i=1,\ldots,n)$ are given by \eqref{T3.1-2}.
As above, $\bold{z}^k:=(z^k_1,\ldots,z^k_n)\ne0$,
${x^k_i-x^k_{0i}}\to0$ $(i=1,\ldots,n)$,
$\eta^k:=\sum_{i=1}^n\|x_i'^{*k}\|\to1$ as $k\to{+\infty}$, $x_i^{*k}:=x_i'^{*k}/\eta^k\in N^F_{\Omega}(x_i^k)$ $(i=1,\ldots,n)$, and conditions \eqref{T3.1P9} are satisfied with $\bold{x}^{*k}:=(x^{*k}_1,\ldots,x^{*k}_n)$.

{To formally complete the proof of condition \GSsal,
it suffices to set $x_i^k:=x_i^{k_0+1}$, $x_i^{*k}:=x_i^{(k_0+1)*}$ $(i=1,\ldots,n)$ and $\bold{x}^{*k}:=\bold{x}^{(k_0+1)*}$
for all $k=1,\ldots,k_0$.}
\qed\end{proof}

The approximate stationarity property (i.e., the case
$\widetilde\al(\sigxz;\siga)=0$)
requires some additional effort.

\begin{corollary}
\label{C3.2}
Let $\sigxz:=\{\bold{x}^k_0\}\subset\widehat\Omega$,
$\bold{x}_0^k:=(x^k_{01},\ldots,x^k_{0n})$ $(k\in\N)$.
Suppose that $\{\Omega_1,\ldots,\Omega_n\}$ is approximately sta\-tionary (particularly, sta\-tionary or extremal)
at $\sigxz$.
Then there exist
$x^k_i\in\Omega_i$,
$x_i^{*k}\in N_{\Omega_i}(x_i^k)$
($i=1,\ldots,n$, $k=1,2,\ldots$) such that $x^k_i-x^k_{0i}\to0$ $(i=1,\ldots,n)$ as $k\to+\infty$, and condition \GSs\ holds true.

Furthermore, if the assumed property holds with respect to some $\siga:=\{\bold{a}^k\}\subset X^n\setminus\{0\}$, $\bold{a}^k:=(a^k_1,\ldots,a^k_n)$ $(k\in\N)$ such that $\bold{a}^k\to0$ as $k\to+\infty$, then there exist $\bold{u}^k:=(u^k_1,\ldots,u^k_n)\in\widehat\Omega$ and $v^k\in X$ $(k\in\N)$
such that $\bold{u}^k-\bold{x}^k_0\to0$ ($\bold{u}^k=\bold{x}^k_0$ if $\{\Omega_1,\ldots,\Omega_n\}$ is sta\-tionary or extremal at $\sigxz$ with respect to $\siga$), $v^k\to0$ as $k\to+\infty$, and condition \PDs\ is satisfied with $\bold{x}^{*k}:=(x^{*k}_1,\ldots,x^{*k}_n)$ and
$\bold{z}^k$ given by \eqref{T3.1-2}.
\end{corollary}

\begin{proof}
By Definitions~\ref{D2.11} and \ref{D2.01}\eqref{D2.01.3},
$\{\Omega_1,\ldots,\Omega_n\}$ is approximately $\al$-sta\-tionary at $\sigxz$ with respect to some $\siga:=\{\bold{a}^k\}\subset X^n\setminus\{0\}$, $\bold{a}^k:=(a^k_1,\ldots,a^k_n)$ $(k\in\N)$ such that $\bold{a}^k\to0$ as $k\to+\infty$ for any $\al>0$.
Let $j\in\N$.
By Theorem~\ref{T3.1},
there exist $x_i^k(j),u_i^k(j)\in\Omega_i$, $x^{*k}_i(j)\in N_{\Omega_i}(x_i^k(j))$ $(i=1,\ldots,n)$ and $v^k(j)\in X$ $(k\in\N)$
such that $x^k_i(j)-x^k_{0i}\to0$, $u^k_i(j)-x^k_{0i}\to0$, $v^k(j)\to0$, $\limsup_{k\to+\infty} \|{\sum_{i=1}^nx_i^{*k}(j)}\|\le1/j$,
$\sum_{i=1}^n\|x^{*k}_i(j)\|=1$, $\bold{z}^k(j):=(z_1^k(j),\ldots,z_n^k(j))\ne0$ for all sufficiently large ${k\in\N}$, and
$\lim_{k\to+\infty}{\ang{\bold{x}^{*k}(j),\bold{z}^k(j)}}/ {\|\bold{z}^k(j)\|}=1$,
where $z_i^k(j):=v^k(j)-x_i^k(j)+u_i^k(j)+a_i^k$ $(i=1,\ldots,n)$ and $\bold{x}^{*k}(j):=(x^{*k}_1(j),\ldots,x^{*k}_n(j))$.
If $\{\Omega_1,\ldots,\Omega_n\}$ is sta\-tionary or extremal at $\sigxz$ with respect to $\siga$, then, in view of Definition~\ref{D2.01}\eqref{D2.01.2}, $u^k_i(j)=x^k_{0i}(j)$ $(i=1,\ldots,n,\; k=1,2,\ldots)$.
Hence, there is an integer $k^j>j$
such that
$\|{\sum_{i=1}^nx_i^{*k}(j)}\|<2/j$
for all integers $k\ge k^j$.
Note that $k^j\to+\infty$ as $j\to+\infty$.

For each $k\in\N$, set $j^k:=\max\{j\in\N\mid k^j\le k\}$ if $k\ge k^1$ and $j^k:=1$ otherwise.
Thus, $j^k\to+\infty$ as $k\to+\infty$.
Setting $x_i^k:=x_i^k(j^k)$, $u_i^k:=u_i^k(j^k)$,
$x_i^{*k}:=x_i^{*k}(j^k)$ $(i=1,\ldots,n)$, $v^k:=v^k(j^k)$ and $\bold{z}^k:=\bold{z}^k(j^k)$, we see that
$x^k_i-x^k_{0i}\to0$, $u^k_i-x^k_{0i}\to0$ $(i=1,\ldots,n)$, $v^k\to 0$, and
$\sum_{i=1}^nx^{*k}_i\to0$ as $k\to+\infty$.
Hence, conditions \GSs\ and \PDs\ are satisfied.
\qed\end{proof}

\begin{remark}
\label{R3.4}
\begin{enumerate}
\item
The generalised se\-paration conditions \GSsal\ and \GSs\ are the main dual necessary conditions
in Theorem~\ref{T3.1} and Corollary~\ref{C3.2}.
The complementary primal-dual condition \PDs\ provides additional restrictions on the dual variables.
\item
In view of Remark~\ref{R2.16}, Theorem~\ref{T3.1} and Corollary~\ref{C3.2} improve \cite[Theorem~3.2 and Corollaries~3.3 and 3.4]{CuoKru}.
\item
In the particular case when $x_{0i}^k=\bx$ for some $\bx\in\bigcap_{i=1}^n\Omega_i$ and all $i=1,\ldots,n$ and $k=1,2,\ldots$, Theorem~\ref{T3.1} and Corollary~\ref{C3.2} give dual necessary conditions for
the specified properties at $\bx$.
Conditions $x^k_i-x^k_{0i}\to0$ and $u^k_i-x^k_{0i}\to0$ reduce in this case to $x_i^k\to\bx$ and $u_i^k\to\bx$, respectively.
With this in mind, Corollary~\ref{C3.2} extends the conventional {extremal principle} \cite{KruMor80,MorSha96,Mor06.1} and its numerous generalisations.

\item
Imposing certain sequential normal compactness assumptions (which are automatically satisfied in finite dimensions), one can formulate limiting versions of Theorem~\ref{T3.1} and Corollary~\ref{C3.2} in terms of certain types of limiting normal cones.
This topic is going to be explored in more detail elsewhere.
\end{enumerate}
\end{remark}

The assertions in the first parts of Theorem~\ref{T3.1} and Corollary~\ref{C3.2} can be partially reversed, i.e., the dual necessary generalised separation conditions discussed in this section are in a sense also sufficient.
To show this, we need to prove first the next lemma which is of independent interest.

\begin{lemma}
\label{L3.5}
Let $\bold{x}:=(x_1,\ldots,x_n)\in\widehat\Omega$, $x^*_i\in N^F_{\Omega_i}(x_i)$ $(i=1,\ldots,n)$, $\|\sum_{i=1}^{n}x^*_i\|<\al$, $\sum_{i=1}^{n}\|x^*_i\|=1$, and $\eps>0$.
Then there exists an
$\bold{a}:=(a_1,\ldots,a_n)\in\eps\B_{X^n}\setminus\{0\}$ such that ${\|\bold{a}\|}/{\zeta(\bold{x}+\bold{a})}<\al$.
\end{lemma}

\begin{proof}
Observe that $\bold{x}^*:=(x^*_1,\ldots,x^*_n)\in N^F_{\widehat\Omega}(\bold{x})$ (see, e.g., \cite[Proposition~2.6]{CuoKru25}).
Choose an $\al'\in(\|\sum_{i=1}^{n}x^*_i\|,\al)$ and set $\xi:=(\al-\al')/2$.
By the definition of \Fr\ normal cone,
there is a $\rho\in(0,\varepsilon/\al)$ such that
\begin{gather}
\label{T4.4P1}
\langle \bold{x}^*,\bold{u}-\bold{x}\rangle\le \dfrac{\xi}{\al+1}\|\bold{u}-\bold{x}\|<\xi\rho
\;\;
\text{for all}
\;\;
\bold{u}\in\widehat\Omega\cap B_{(\al+1)\rho}(\bold{x}).
\end{gather}
Choose a nonzero $\bold{a}:=(a_1,\ldots,a_n)\in\al\rho\B_{X^n}$ ($\subset\eps\B_{X^n}$) such that
\begin{gather}
\label{T4.4P2}
\langle \bold{x}^*,\bold{a}\rangle >(\al-\xi)\rho=(\al'+\xi)\rho.
\end{gather}
Suppose that $\zeta(\bold{x}+\bold{a})<\rho$.
Then $u_i-x_i-a_i=x_0$ for some $x_0\in\rho\B_X$ and $\bold{u}:= (u_1,\ldots,u_n)\in\widehat\Omega$, and all $i=1,\ldots,n$.
Hence, $\|\bold{u}-\bold{x}\|\le\|\bold{a}\|+\|x_0\|<(\al+1)\rho$,
and, by \eqref{T4.4P1},
$\ang{\bold{x}^*,\bold{u}-\bold{x}}<\xi\rho$.
Employing \eqref{T4.4P2}, we obtain
\begin{align*}
-\al'\rho<-\Big\|\sum_{i=1}^{n}x^*_i\Big\|\|x_0\| \le\Big\langle \sum_{i=1}^{n}x^*_i,x_0\Big\rangle=\langle \bold{x}^*,\bold{u}-\bold{x}\rangle-\langle \bold{x}^*,\bold{a}\rangle<-\al'\rho.
\end{align*}
The contradiction proves that $\zeta(\bold{x}+\bold{a})\ge\rho$.
Hence, ${\|\bold{a}\|}/{\zeta(\bold{x}+\bold{a})}<\al$.
\qed\end{proof}

As a consequence of Lemma~\ref{L3.5}, we have the promised statement partially reversing the assertions in the first parts of Theorem~\ref{T3.1} and Corollary~\ref{C3.2}.

\begin{proposition}
\label{P3.6}
Let $\sigx:=\{\bold{x}^k\}\subset\widehat\Omega$,
$\bold{x}^k:=(x^k_1,\ldots,x^k_n)$, and
$x^{*k}_i\in N^F_{\Omega_i}(x^k_i)$ ($i=1,\ldots,n$, $k=1,2,\ldots$).
Then \GSs\ \folgt \Ss\, and
\GSsal\ \folgt \Ssxal\; for any $\al>0$.
\end{proposition}

\begin{proof}
Let $\varepsilon^k>0$ $(k\in\N)$,
$\eps^k\downarrow0$ as $k\to+\infty$, and condition \GSsal\ be satisfied.
Choose an $\al'<\al$ so that $\|{\sum_{i=1}^nx^{*k}_i}\|<\al'$ for all sufficiently large $k\in\N$.
By Lemma~\ref{L3.5},
for each such $k$, there exists  an
$\bold{a}^k:=(a^k_1,\ldots,a^k_n)\in\eps^k\B_{X^n}\setminus\{0\}$  such that $\|\bold{a}^k\|/{\zeta(\bold{x}^k+\bold{a}^k)}<\al'$.
Hence, $\bold{a}^k\to0$ and $\al(\sigx;\siga)\le\al'<\al$, where
$\siga:=\{\bold{a}^k\}$, i.e., $\{\Omega_1,\ldots,\Omega_n\}$ is $\al$-sta\-tionary at
$\sigx$.
If condition \GSs\ is satisfied, then condition \GSsal\ is satisfied for any $\al>0$, and the conclusion follows.
\qed\end{proof}

Combining Theorem~\ref{T3.1}, Corollary~\ref{C3.2} and Proposition~\ref{P3.6}, we obtain the next theorem.

\begin{theorem}
[Sequential extended extremal principle]
\label{T3.9}
Let $\al>0$.
Suppose that either $X$ is Asplund or $\Omega_1,\ldots,\Omega_n$ are convex.
The collection
$\{\Omega_1,\ldots,\Omega_n\}$ is approximately $\al$-sta\-tionary (approximately stationary) at $\sigxz:=\{\bold{x}^k_0\}\subset\widehat\Omega$, $\bold{x}_0^k:=(x^k_{01},\ldots,x^k_{0n})$ if and only if
there exist
$\bold{x}^k:=(x^k_1,\ldots,x^k_n)\in\widehat\Omega$ $(k\in\N)$ and
$x_i^{*k}\in N_{\Omega_i}(x_i^k)$
($i=1,\ldots,n$, $k=1,2,\ldots$) such that $\bold{x}^k-\bold{x}^k_0\to0$ as $k\to+\infty$, and condition \GSsal\ (condition \GSs) holds true.
\end{theorem}

\begin{proof}
Under the assumptions, thanks to our standing convention, $N$ is the \Fr\ normal cone.
The necessity follows from Theorem~\ref{T3.1} (Corollary~\ref{C3.2}).
The converse implication
is a consequence of Proposition~\ref{P3.6}, which yields $\al$-sta\-tionarity (sta\-tionarity) of $\{\Omega_1,\ldots,\Omega_n\}$ at $\sigx:=\{\bold{x}^k\}$.
It remains to notice that, thanks to Proposition~\ref{P2.14}, $\al$-sta\-tionarity (sta\-tionarity) at
$\sigx$ implies approximate $\al$-sta\-tionarity (approximate sta\-tionarity) at
$\sigxz$.
\qed\end{proof}

\begin{remark}
\begin{enumerate}
\item
Theorem~\ref{T3.9} generalises and extends Theorem~\ref{T1.4} as well as \cite[Theorem~3.6 and Corollary~3.7]{CuoKru} and \cite[Theorem~5]{BuiKru18}.
\item
Theorem~\ref{T3.9} can be reformulated as a statement providing dual characterisations of the absence of approximate $\al$-sta\-tionarity or approximate stationarity.
These properties can be interpreted as a kind of (sequential) \emph{transversality} of collections of sets; cf. \cite{Kru05,KruTha13,CuoKru21.2,CuoKru25}.
This topic is going to be explored in more detail elsewhere.
\item
It is easy to see from the proofs that in Lemma~\ref{L3.5} and Proposition~\ref{P3.6}, $X$ can be an arbitrary normed space, and the sets $\Omega_1,\ldots,\Omega_n$ do not have to be closed.
As a consequence, the sufficiency in Theorem~\ref{T3.9} is also true in this less restrictive setting.
The Banach structure of the space and closedness of the sets are only needed for the necessity part, the proof of which relies on Theorem~\ref{T3.1} and Corollary~\ref{C3.2}.
\end{enumerate}
\end{remark}


\section{Sequential necessary optimality and stationarity conditions}\label{S4}

In this section, we consider the following minimisation problem with a single geometric constraint:
\begin{gather}
	\label{P}
	\tag{${P}$}
	\text{minimize }\;f(x)\quad \text{subject to }\; x\in\Omega,
\end{gather}
where $\Omega$ is a nonempty
subset of a normed vector space $X$ and
$f:X\to\R_\infty$.

Let a sequence $\sigma_x:=\{x^k\}\subset\Omega\cap\dom f$ be given.
The next two definitions extend the notion of minimising sequence. Below, $S_f(\mu)$ denotes the lower level set of $f$ at level $\mu\in\R$, i.e., $S_f(\mu):=\{x\in X\mid f(x)\le\mu\}$, and, given an $x\in X$ and an $\eps>0$, we use the notation $\zeta_{f,\Omega}(x,\eps):=d\big(x,\Omega\cap S_f(f(x)-\eps)\big)$.

\begin{definition}
\label{D4.1}
Let $\sigma_\eps:=\{\eps^k\}\subset\R_+\setminus\{0\}$ and $\eps^k\downarrow0$ as $k\to+\infty$.
The sequence $\sigma_x$ is
\begin{enumerate}
\item
\label{D4.1.1}
$\sigma_\eps$-extremal for problem \eqref{P} if
\begin{gather}
\label{D4.1-01}
\rho_{f,\Omega}(\sigma_x;\sigma_\eps) :=\liminf_{k\to+\infty}
\zeta_{f,\Omega}(x^k,\eps^k)>0;
\end{gather}
if $\rho_{f,\Omega}(\sigma_x;\sigma_\eps)=+\infty$, we say that $\sigma_x$ is globally $\sigma_\eps$-extremal for problem \eqref{P};
\item
\label{D4.1.2}
$(\sigma_\eps,\al)$-stationary for problem \eqref{P} if
\begin{gather}
\label{D4.1-02}
\al_{f,\Omega}(\sigma_x;\sigma_\eps):=\limsup_{k\to+\infty}
\frac{\eps^k} {\zeta_{f,\Omega}(x^k,\eps^k)}<\al<+\infty;
\end{gather}
if $\al_{f,\Omega}(\sigma_x;\sigma_\eps)=0$ (hence, \eqref{D4.1-02} is satisfied with any $\al>0$), we say that $\sigma_x$ is $\sigma_\eps$-sta\-tionary for problem \eqref{P};
\item
\label{D4.1.3}
approximately $(\sigma_\eps,\al)$-stationary for problem \eqref{P} if
\begin{gather}
\label{D4.1-03}
\widetilde\al_{f,\Omega}(\sigma_x;\sigma_\eps):= \inf_{\substack{
\sigma_u:=\{u^k\}\subset\Omega,\,u^k-x^k\to0,\, f(u^k)-f(x^k)\to0}}
\al_{f,\Omega}(\sigma_u;\sigma_\eps)<\al<+\infty;
\end{gather}
if $\widetilde\al_{f,\Omega}(\sigma_x;\sigma_\eps)=0$ (hence, \eqref{D4.1-03} is satisfied with any $\al>0$), we say that $\sigma_x$ is approximately $\sigma_\eps$-stationary for problem \eqref{P}.
\end{enumerate}
\end{definition}

\begin{definition}
\label{D4.2}
The sequence $\sigma_x$ is
extremal (resp., globally extremal, $\al$-sta\-tionary, stationary, approximately $\al$-stationary, approximately stationary) for problem \eqref{P} if it is $\sigma_\eps$-ex\-tremal (resp., globally $\sigma_\eps$-extremal, $(\sigma_\eps,\al)$-sta\-tionary, $\sigma_\eps$-stationary, approximately $(\sigma_\eps,\al)$-sta\-tionary, approximately $\sigma_\eps$-stationary) for problem \eqref{P} with some $\sigma_\eps:=\{\eps^k\}\subset\R_+\setminus\{0\}$ such that $\eps^k\downarrow0$.
\end{definition}

\begin{remark}
The relationships between the respective properties in Definitions~\ref{D4.1} and \ref{D4.2} are straightforward:
$\sigma_\eps$-ext\-remality~$\Rightarrow$ $\sigma_\eps$-stationarity~$\Rightarrow$ approximate
$\sigma_\eps$-sta\-tio\-na\-ri\-ty, and
ext\-remality~$\Rightarrow$ stationarity~$\Rightarrow$ approximate
sta\-tio\-na\-ri\-ty.
\end{remark}

\begin{proposition}
\label{P5.3}
Let $\inf_\Omega f>-\infty$.
If $\sigma_x$ is a minimising sequence of $f$ on $\Omega$, i.e., $f(x^k)\to\inf_\Omega f$ as $k\to+\infty$, then it is globally extremal for problem \eqref{P}.
\end{proposition}

\begin{proof}
Let $\sigma_x$ be a minimising sequence of $f$ on $\Omega$.
Then there exists a sequence $\sigma_\eps:=\{\eps^k\}\subset\R_+\setminus\{0\}$ with $\eps^k\downarrow0$ such that $f(x^k)<\inf_\Omega f+\eps^k$ for all
$k\in\N$.
The last inequality yields $f(x)>f(x^k)-\eps^k$ for all $x\in\Omega$; hence, $\Omega\cap S_f(f(x^k)-\eps^k)=\es$, and consequently, $\rho_{f,\Omega}(\sigma_x;\sigma_\eps)=+\infty$, i.e., $\sigma_x$ is globally extremal for problem \eqref{P}.
\qed\end{proof}

\begin{example}\label{E4.3}
Let $X$ be a normed space, $\Omega:=X$,
$f(x):=-\|x\|^2$ $(x\in X)$, $x^k:=0$ $(k\in\N)$, and $\eps^k\downarrow0$.
Then $S_f(-\eps^k)=\{x\in X\mid \|x\|\ge\sqrt{\eps^k}\}$ $(k\in\N)$, and consequently, $\rho_{f,\Omega}(\sigma_x;\sigma_\eps) =\lim_{k\to+\infty}\sqrt{\eps^k}=0$ and
$\al_{f,\Omega}(\sigma_x;\sigma_\eps)=\lim_{k\to+\infty} \frac{\eps^k}{\sqrt{\eps^k}} =\lim_{k\to+\infty}\sqrt{\eps^k}=0$.
Hence, $\sigma_x$ is not $\sigma_\eps$-extremal, but is $\sigma_\eps$-stationary (hence, also stationary) for problem \eqref{P}.
\end{example}

The next proposition gives some characterisations of the properties in Definition~\ref{D4.2}.

\begin{proposition}
\label{P4.5}
The sequence $\sigma_x$ is
\begin{enumerate}
\item
extremal for problem \eqref{P} if and only if there exists a $\rho>0$ such that
\begin{gather}
\label{P4.5-1}
\lim_{k\to+\infty}\Big(f(x^k)-\inf_{\Omega\cap B_\rho(x^k)}f\Big)=0;
\end{gather}
\item
globally extremal for problem \eqref{P} if and only if condition \eqref{P4.5-1} is satisfied for all $\rho>0$;

\item
\label{P4.5.3}
$\al$-stationary for problem \eqref{P} if and only if
\begin{gather}
\label{P4.5-3}
\al_{f,\Omega}^\circ(\sigma_x):= \inf_{\rho^k\downarrow0}\limsup_{k\to+\infty} \frac{f(x^k)-\inf_{\Omega\cap B_{\rho^k}(x^k)}f}{\rho^k}<\al;
\end{gather}

\item
stationary for problem \eqref{P} if and only if $\al_{f,\Omega}^\circ(\sigma_x)=0$;

\item
approximately $\al$-stationary for problem \eqref{P} if and only if
\begin{gather*}
\widetilde\al_{f,\Omega}^\circ(\sigma_x) :=\inf_{\substack{
\{u^k\}\subset\Omega,\,u^k-x^k\to0,\, f(u^k)-f(x^k)\to0}} \al_{f,\Omega}^\circ(\sigma_u)<\al;
\end{gather*}

\item
approximately stationary for problem \eqref{P} if and only if $\widetilde\al_{f,\Omega}^\circ(\sigma_x)=0$.
\end{enumerate}
\end{proposition}

\begin{proof}
\begin{enumerate}
\item
Let $\sigma_x$ be extremal for problem \eqref{P}.
By Definitions~\ref{D4.1} and \ref{D4.2}, there exists a sequence $\sigma_\eps:=\{\eps^k\}\subset\R_+\setminus\{0\}$ with $\eps^k\downarrow0$ such that condition \eqref{D4.1-01} is satisfied, and consequently, $\zeta_{f,\Omega}(x^k,\eps^k)>\rho$ for any $\rho\in\big(0,\rho_{f,\Omega}(\sigma_x;\sigma_\eps)\big)$ and all sufficiently large $k\in\N$.
The latter inequality yields $\Omega\cap S_f(f(x^k)-\eps^k)\cap B_\rho(x^k)=\es$, and consequently, $f(x^k)<+\infty$ and $f(x)>f(x^k)-\eps^k$ for all $x\in\Omega\cap B_\rho(x^k)$.
Hence, $0\le f(x^k)-\inf_{\Omega\cap B_\rho(x^k)}f\le\eps^k$, and consequently, condition \eqref{P4.5-1} is satisfied.

Conversely, let $\rho>0$, and condition \eqref{P4.5-1} be satisfied.
Then there exists a sequence $\sigma_\eps:=\{\eps^k\}\subset\R_+\setminus\{0\}$ with $\eps^k\downarrow0$ such that $f(x^k)-\inf_{\Omega\cap B_\rho(x^k)}f<\eps^k$ for all sufficiently large $k\in\N$.
The latter inequality yields $f(x)>f(x^k)-\eps^k$ for all $x\in\Omega\cap B_\rho(x^k)$, and consequently, $\zeta_{f,\Omega}(x^k,\eps^k)\ge\rho$ $(k\in\N)$.
Hence, $\rho_{f,\Omega}(\sigma_x;\sigma_\eps)\ge\rho>0$, i.e., $\sigma_x$ is $\sigma_\eps$-extremal for problem \eqref{P}.

\item
The above proof is applicable to the case of global extremality.

\item
Let $\sigma_x$ be $\al$-sta\-tionary for problem \eqref{P}.
By Definitions~\ref{D4.1} and \ref{D4.2}, there exists a sequence $\sigma_\eps:=\{\eps^k\}\subset\R_+\setminus\{0\}$ with $\eps^k\downarrow0$ such that condition \eqref{D4.1-02} is satisfied, and consequently, there is an $\al'\in(0,\al)$ such that $\zeta_{f,\Omega}(x^k,\eps^k)>\rho^k:=\eps^k/\al'$ for all sufficiently large $k\in\N$.
As above, the latter inequality implies $0\ge\inf_{\Omega\cap B_{\rho^k}(x^k)}f-f(x^k)\ge -\eps^k=-\al'\rho^k$.
Hence,
\begin{gather*}
\limsup_{k\to+\infty}\frac{f(x^k)-\inf_{\Omega\cap B_{\rho^k}(x^k)}f}{\rho^k}\le\al'<\al,
\end{gather*}
which yields \eqref{P4.5-3}.

Conversely, let condition \eqref{P4.5-3} be satisfied.
Then there exist a number $\al'\in(0,\al)$ and a sequence $\rho^k\downarrow0$ such that, for all sufficiently large $k\in\N$, we have $f(x^k)-\inf_{\Omega\cap B_{\rho^k}(x^k)}f<\eps^k:=\al'{\rho^k}$, and consequently, $\zeta_{f,\Omega}(x^k,\eps^k)\ge\rho^k$.
Denote $\sigma_\eps:=\{\eps^k\}$.
Thus, $\eps^k\downarrow0$, and
\begin{gather*}
\limsup_{k\to+\infty}\frac{\eps^k}{\zeta_{f,\Omega}(x^k,\eps^k)}\le\al'<\al,
\end{gather*}
i.e., $\sigma_x$ is $(\sigma_\eps,\al)$-sta\-tionary for problem \eqref{P}.
\end{enumerate}
The last three assertions are consequences of assertion \eqref{P4.5.3}.
\qed\end{proof}

\begin{remark}
The characterisations in Proposition~\ref{P4.5} give clear relationships between the extremality and stationarity properties in Definition~\ref{D4.2} and the corresponding ones studied in \cite[Sections~4 and 5]{CuoKru}.
Note that the definitions in \cite{CuoKru} assume that $\{f(x^k)\}$ converges to some number $\mu_0$.
We do not use this assumption here.

Under this superfluous assumption, a sequence $\sigma_x:=\{x^k\}\subset\Omega$ is firmly
$\inf$-stationary (resp., $\inf$-sta\-tionary, approximately $\inf$-stationary) for problem \eqref{P} at level $\mu_0$ in the sense of \cite[Definition~4.5]{CuoKru} if and only if some subsequence of $\sigma_x$ is extremal (resp., stationary, approximately stationary) for problem \eqref{P}.
Thus, the definitions adopted in the current paper
are more general and more specific in identifying the sequences relevant for the properties.
If $x^k=\bx$ for all $k\in\N$, the corresponding properties coincide.
If, additionally, $\Omega=X$, the stationarity and approximate stationarity properties reduce to the, respectively, $\inf$-stationarity and approximate $\inf$-stationarity (weak $\inf$-stationarity) studied in \cite{Kru06.2,Kru09}.
The sequential examples in \cite[Examples~4.7--4.9]{CuoKru} complement the `at-a-point' Example~\ref{E4.3} and provide additional illustrations of the properties.

In this paper, for brevity we do not use the `$\inf$-terminology' from \cite{CuoKru}.
This does not lead to confusion as we do not consider maximisation problems here.
This way, we keep the terminology consistent with that used in Sections~\ref{S2} and \ref{S3}.
Of course, a point of local maximum cannot be extremal in the sense of Definition~\ref{D4.2} (as illustrated by Example~\ref{E4.3}).
\end{remark}

To embed problem \eqref{P} into the model studied in Section~\ref{S2},
we consider
the sets
\begin{gather}
\label{Om}
\Omega_1:=\epi f,\quad
\Omega_2:=\Omega\times(-\infty,0]\AND
\widehat\Omega:=\Omega_1\times\Omega_2,
\end{gather}
and sequences
\begin{gather}
\label{S4-32}
x_1^k:=(x^k,f(x^k))\in\Omega_1,\quad x_2^k:=(x^k,0)\in\Omega_2\AND \bold{x}^k:=(x_1^k,x_2^k)\quad(k\in\N),
\\
\label{S4-28}
a_1^k:=(0,0),\quad a_2^k:=(0,\eps^k)\AND \bold{a}^k:=(a_1^k,a_2^k)\quad(k\in\N),
\end{gather}
where $\{\eps^k\}\subset\R_+\setminus\{0\}$ and $\eps^k\downarrow0$.
Thus, for all $k\in\N$, we have $\|\bold{a}^k\|=\eps^k$,
\begin{gather*}
\Omega_1-x_1^k-a_1^k=\{(x-x^k,\mu)\mid x\in X,\; f(x)-f(x^k)\le\mu\},\\
\Omega_2-x_2^k-a_2^k=\{(x-x^k,\mu)\mid x\in\Omega,\; \mu\le-\eps^k\},
\\
(\Omega_1-x_1^k-a_1^k)\cap(\Omega_2-x_2^k-a_2^k) =\{(x-x^k,\mu)\mid x\in\Omega,\; f(x)-f(x^k)\le\mu\le-\eps^k\},
\\
\zeta(\bold{x}^k+\bold{a}^k) =\max\{d\big(0,\Omega\cap S_f(f(x^k)-\eps^k)\big),\eps^k\} =\max\{\zeta_{f,\Omega}(x^k,\eps^k),\eps^k\}.
\end{gather*}

Denote $\sigx:=\{\bold{x}^k\}$, $\siga:=\{\bold{a}^k\}$ and $\sigma_\eps:=\{\eps^k\}$.
In view of the above observations, it follows from the definitions in \eqref{D2.01-1}, \eqref{D2.01-2}, \eqref{D2.01-3}, \eqref{D4.1-01}, \eqref{D4.1-02} and \eqref{D4.1-03}
that
\begin{align*}
\notag
\rho(\sigx;\siga) &=\liminf_{k\to+\infty} \max\{\zeta_{f,\Omega}(x^k,\eps^k),\eps^k\}
=\liminf_{k\to+\infty} \zeta_{f,\Omega}(x^k,\eps^k)=\rho_{f,\Omega}(\sigma_x;\sigma_\eps),
\\
\notag
\al(\sigx;\siga) &=\limsup_{k\to+\infty}
\frac{\eps^k} {\max\{\zeta_{f,\Omega}(x^k,\eps^k),\eps^k\}}
\\&
=\min\Big\{\limsup_{k\to+\infty}
\frac{\eps^k} {\zeta_{f,\Omega}(x^k,\eps^k)},1\Big\}
=\min\{\al_{f,\Omega}(\sigma_x;\sigma_\eps),1\},
\\
\widetilde\al(\sigx;\siga)
&=\inf_{\sigma_{\bold{u}}:=\{\bold{u}^k\}\in\Omega_1\times\Omega_2,\;\bold{u}^k-\bold{x}^k\to 0}\al(\sigma_{\bold{u}};\siga)\\
&\le \inf_{\substack{\sigma_{u}:=\{u^k\}\in\Omega,\;u^k-x^k\to 0,\\f(u^k)- f(x^k)\to 0}} \min\{\al_{f,\Omega}(\sigma_u;\sigma_\eps),1\}
=\min\{\widetilde\al_{f,\Omega}(\sigma_x;\sigma_\eps),1\}.
\end{align*}

We can now formulate relationships between the properties in Definitions~\ref{D4.1} and \ref{D4.2}, and the corresponding ones in Definitions~\ref{D2.01} and \ref{D2.11}.

\begin{proposition}
\label{P4.3}
Let $\Omega_1$, $\Omega_2$, $x_1^k$, $x_2^k$, $a_1^k$, $a_2^k$, $\bold{x}^k$ and $\bold{a}^k$ $(k\in\N)$ be given by \eqref{Om}, \eqref{S4-32} and \eqref{S4-28}, and $\al\in(0,1)$.
The sequence $\sigma_x$ is
$\sigma_\eps$-extremal (resp., $(\sigma_\eps,\al)$-stationary, $\sigma_\eps$-sta\-tionary) for problem \eqref{P} if and only if $\{\Omega_1,\Omega_2\}$ is extremal (resp., $\al$-stationary, stationary) at $\sigx$ with respect to $\siga$.

If $\sigma_x$ is approximately $(\sigma_\eps,\al)$-stationary (resp., approximately $\sigma_\eps$-stationary) for problem \eqref{P}, then $\{\Omega_1,\Omega_2\}$ is approximately $\al$-stationary (resp., approximately stationary) at $\sigx$ with respect to $\siga$.

As a consequence, if $\sigma_x$ is
extremal (resp., $\al$-stationary, sta\-tionary, approximately $\al$-sta\-tionary, approximately stationary) for problem \eqref{P}, then $\{\Omega_1,\Omega_2\}$ is extremal (resp., $\al$-sta\-tionary, stationary, approximately $\al$-stationary, approximately stationary) at $\sigx$.
\end{proposition}

In the next theorem and the rest of the paper, we suppose that $X$ is Banach, $f$ is \lsc\ and $\Omega$ is closed.
As a consequence, the sets $\Omega_1$ and $\Omega_2$ given by \eqref{Om} are closed.
Recall our standing convention that the subdifferential and normal cone generic notations $\sd$ and $N$ stand for $\sd^{C}$ and $N^C$ if $X$ is a general Banach space, and for $\sd^F$ and $N^F$ if $X$ is Asplund.

\begin{theorem}
\label{T4.5}
Suppose that $\sigma_x$ is approximately stationary (particularly, sta\-tionary or extremal)
 for problem \eqref{P}.
Then
there exist $(x^k_1,\mu^k_1)\in\epi f$, $(x_1^{*k},\nu^k_1)\in N_{\epi f}(x^k_1,\mu^k_1)$, $x^k_2\in\Omega$,
{$\mu^k_2\le0$,}
$x_2^{*k}\in N_{\Omega}(x^k_2)$, and $\nu^k_2\ge0$ such that $\nu^k_2\mu^k_2=0$ $(k\in\N)$,
\begin{gather}
\label{T4.5-1}
x^k_1-x^k\to 0,\;x^k_2-x^k\to 0,\;
\mu^k_1-f(x^k)\to0,\;
{\mu^k_2\to0\;\;}
\text{as}\;\;k\to+\infty,\\
\label{T4.5-2}
x^{*k}_1+x^{*k}_2\to0,\quad
{\nu^k_1+\nu^k_2\to0\;\;}
\text{as}\;\;k\to+\infty,\\
\label{T4.5-4}
\|x^{*k}_1\|+|\nu^k_1|=1\;\; (k\in\N).
\end{gather}
\end{theorem}

\begin{proof}
By Proposition~\ref{P4.3}, the pair $\{\Omega_1,\Omega_2\}$ given by \eqref{Om} is approximately stationary at $\sigx$.
The statement follows from Corollary~\ref{C3.2}
applied with $n=2$: there exist
$(x^k_1,\mu^k_1)\in\epi f$, $(x^k_2,\mu^k_2)\in\Omega_2$, $(x_1^{*k},\nu^k_1)\in N_{\epi f}(x^k_1,\mu^k_1)$, and $(x_2^{*k},\nu^k_2)\in N_{\Omega_2}(x^k_2,\mu^k_2)$ such that conditions
\eqref{T4.5-1} and \eqref{T4.5-2} are satisfied, and
$\|\big((x_1^{*k},\nu^k_1),(x_2^{*k},\nu^k_2)\big)\|=1$ $(k\in\N)$.
Then $x_2^{*k}\in N_{\Omega}(x^k_2)$, $\nu^k_2\ge0$, $\nu^k_2\mu^k_2=0$ $(k\in\N)$, and
\begin{gather*}
\lim_{k\to+\infty} (\|x_1^{*k}\|+|\nu^k_1|)=
\lim_{k\to+\infty} (\|x_2^{*k}\|+\nu^k_2)=
\frac12.
\end{gather*}
We can assume without loss of generality that $(x_1^{*k},\nu^k_1)\ne0$ for all $k\in\N$.
Replacing normal vectors $(x_1^{*k},\nu^k_1)$ and $(x_2^{*k},\nu^k_2)$ by the `stretched' normal vectors $(x_1^{*k},\nu^k_1)/(\|x_1^{*k}\|+|\nu^k_1|)$ and $(x_2^{*k},\nu^k_2)/(\|x_1^{*k}\|+|\nu^k_1|)$, respectively, while keeping the original notation, we can ensure condition \eqref{T4.5-4}.
Conditions \eqref{T4.5-2} remain true.
\qed\end{proof}	

\begin{remark}
\label{R4.6}
Since vector $(x^{*k}_1,\nu^k_1)$ in Theorem~\ref{T4.5} is normal to $\epi f$, it obviously holds $\nu_1^k\le0$ and $\nu_1^k(\mu^k_1-f(x^k_1))=0$ for all $k\in\N.$
\end{remark}	
\if{
\NDC{30/12/25.
$\{\varepsilon^k\}$ does not involve in the conlcusion!
We have $\mu^k_2-\varepsilon^k\to\mu_0$, but this condition does not give any useful information.
}
\AK{8/01/26. The above theorem only uses the \GSs\ condition. We must make \PDs\ work too. Some my experiments are in Sect. 5.1. No definite conclusions at the moment.}

\AK{5/01/26.
This is an important observation!
If we stop here, we can remove $\{\varepsilon^k\}$ from the definition (by writing ``for some $\eps^k$''; thus, basically using Corollary~\ref{C3.3}).
However, I am hoping to get some two-parameter statement using the full power of Theorem~\ref{T3.1} and make condition \PDs\ work.
Recall that $\eps^k$ is part of the perturbation vector.
We will need to think also about adopting proper terminology.

This brings us back to the key question: which stationarity property we should choose as the main one?
At the moment, I see the main aim of this section as to help us answer this question.}
\AK{5/01/26.
Could it make sense to reformulate the corollary as a sufficient condition and possibly add it to the theorem?}
\AK{20/01/26.
I am not sure it can make sense to formulate a corollary in the form of a ``theorem of alternative'' again.
At the same time, some sufficient (qualification) conditions are needed, i.e., sufficient conditions for item (iii) in the theorem (if it is correct).
Perhaps, item (iii) should be removed from the theorem.
If the theorem is of value, it could make sense to formulate a series of corollaries.}
}\fi

Theorem~\ref{T4.5} yields the following theorem-of-alternative type statement.

\begin{corollary}
\label{C4.11}
Suppose that $\sigma_x$ is approximately stationary (particularly, sta\-tionary or extremal) for problem \eqref{P}.
Then one of the following assertions holds true:
\begin{enumerate}
\item
\label{C4.11.1}
there exist an $M>0$, $x^k_1\in X$, $x^k_2\in\Omega$, $x^{*k}_1\in\sd f(x^k_1)$ and $x^{*k}_2\in N_{\Omega}(x^k_2)\cap(M\B_{X^*})$ $(k\in\N)$ such that $f(x^k_1)-f(x^k)\to0$,
\begin{gather}
\label{C4.11-1}
x^k_1-x^k\to 0,\;\; x^k_2-x^k\to 0\AND x^{*k}_1+x^{*k}_2\to0\;\;
\text{as}\;\;k\to+\infty;
\end{gather}

\item
\label{C4.11.2}
there exist $(x^k_1,\mu^k_1)\in\epi f$, $x^k_2\in\Omega$, $\mu^k_2\le0$, $(x_1^{*k},\nu^k_1)\in N_{\epi f}(x^k_1,\mu^k_1)$,  $x_2^{*k}\in N_{\Omega}(x^k_2)$, and $\nu^k_2\ge0$ $(k\in\N)$ such that $\nu^k_2\mu^k_2=0$ $(k\in\N)$, conditions \eqref{C4.11-1} are satisfied,
\begin{gather}
\label{C4.11-8}
\mu^k_1-f(x^k)\to0,\;\; \mu^k_2\to0,\;\; \nu^k_1\to0,\;\; \nu^k_2\to0\AND
\|x^{*k}_1\|=\|x^{*k}_2\|=1\;\;(k\in\N).
\end{gather}
\end{enumerate}
\end{corollary}

\begin{proof}
By Theorem~\ref{T4.5}, there exist $(x^k_1,\mu^k_1)\in\epi f$, $x^k_2\in\Omega$, $\mu^k_2\le0$, $(x_1^{*k},\nu^k_1)\in N_{\epi f}(x^k_1,\mu^k_1)$, $x_2^{*k}\in N_{\Omega}(x^k_2)$, and $\nu^k_2\ge0$ $(k\in\N)$ such that $\nu^k_2\mu^k_2=0$ $(k\in\N)$, and conditions \eqref{T4.5-1}--\eqref{T4.5-4} are satisfied.
In the rest of the proof, without further mentioning, all claims hold `for all sufficiently large $k\in\N$'.
Denote $\ga:=\liminf_{k\to+\infty}|\nu^k_1|$.
Two cases are possible leading to the two assertions in the theorem.

\underline{$\ga>0$}.
It follows from \eqref{T4.5-2} that $\liminf_{k\to+\infty}\nu^k_2=\ga>0$.
Hence, $\nu^k_1<0$ and $\nu^k_2>0$,
and consequently, $\mu^k_1=f(x^k_1)$ and $\mu^k_2=0$.
Thus, $f(x^k_1)-f(x^k)\to0$, and conditions \eqref{C4.11-1} are satisfied.
Denote $M:=\ga\iv$.
By \eqref{T4.5-2} and \eqref{T4.5-4},
$$\limsup_{k\to+\infty}\|x^{*k}_1\|/|\nu^k_1|<M\AND \limsup_{k\to+\infty}\|x^{*k}_2\|/\nu^k_2<M.$$
It remains to observe that $x'^{*k}_1:=x^{*k}_1/|\nu^k_1|\in\sd f(x^k_1)$ and $x'^{*k}_2:=x^{*k}_2/\nu^k_2\in N_{\Omega}(x^k_2)\cap (M\B_{X^*})$.

\underline{$\ga=0$}.
Passing to subsequences while keeping the original notation, thanks to \eqref{T4.5-2} and \eqref{T4.5-4}, we have $\nu_1^k\uparrow0$, $\nu_2^k\downarrow0$, $\|x^{*k}_1\|\to1$ and $\|x^{*k}_2\|\to1$.
Hence, $|\nu_1^k|/\|x^{*k}_1\|\to0$ and $\nu_2^k/\|x^{*k}_2\|\to0$.
Scaling the normal vectors, we can assume without loss of generality that $\|x^{*k}_1\|=\|x^{*k}_2\|{=1}$.
Thus, conditions \eqref{C4.11-8} are satisfied.
\qed\end{proof}

\begin{remark}
\label{R4.11}
\begin{enumerate}
\item
Thanks to Proposition~\ref{P5.3}, the conclusions of Theorem~\ref{T4.5} and Corollary~\ref{C4.11} hold true, in particular, if $\inf_\Omega f>-\infty$, and $\sigma_x$ is a conventional minimising sequence.
In this case, Theorem~\ref{T4.5} and Corollary~\ref{C4.11} improve \cite[Theorem~5.2 and Corollary~5.3]{CuoKru}.
\item
The arguments in the proof of Corollary~\ref{C4.11}\eqref{C4.11.1} lead to a seemingly stronger than $x^{*k}_1\in\sd f(x^k_1)$ condition $x^{*k}_1\in\sd f(x^k_1)\cap(M\B_{X^*})$.
It is not difficult to see that, thanks to $\|x^{*k}_2\|<M$ and $x^{*k}_1+x^{*k}_2\to0$, the latter condition is implied by $x^{*k}_1\in\sd f(x^k_1)$ (with possibly a larger~$M$).
\item
Corollary~\ref{C4.11}\eqref{C4.11.1} yields a kind of multiplier rule:
\begin{gather*}
d\big(0,\sd f(x^k_1)+N_{\Omega}(x^k_2)\cap (M\B_{X^*})\big)\to 0\;\;\text{as}\;\;k\to+\infty.
\end{gather*}
The limiting version (at an abstract infinity point) of this rule in finite dimensions is given in \cite[Theorem~6.1]{KimNguPha25}.
\item
\label{R4.11.7}
Part \eqref{C4.11.2} of Corollary~\ref{C4.11} corresponds to `singular' behaviour of $f$ on $\Omega$ with the normal vectors $(x_1^{*k},\nu^k_1)$ to the epigraph of $f$ being `almost horizontal'.
If $\mu_1^k>f(x_1^k)$, then $\nu_1^k=0$ and $x_1^{*k}$ is normal to $\dom f$ at $x_1^k$.
\emph{Qualification conditions} can be used to exclude the singular behaviour in Corollary~\ref{C4.11}\eqref{C4.11.2}, thus, leaving the multiplier rule type assertion in part \eqref{C4.11.1} as the only possibility.
The simplest sufficient conditions of this kind are: 1)~$f$ is Lipschitz continuous, and 2) $\Omega=X$.
In the first case, condition $\nu^k_1\to0$ implies $x_1^{*k}\to0$, while in the second, we have $x_2^{*k}=0$.
Thus, in both cases one of the equalities in \eqref{C4.11-8} is violated.
The mentioned sufficient conditions are rather strong and far from being necessary.
More subtle qualification conditions can be found in \cite[Section~5]{CuoKru}.
\end{enumerate}
\end{remark}	

\begin{example}
Let $X$ be a Banach space, $\Omega:=X$,
$f(x):=-\|x\|^2$ $(x\in X)$, and $x^k:=0$.
The `sequence' $\sigma_x:=\{x^k\}$ is stationary for problem \eqref{P} (see Example~\ref{E4.3}).
In view of Remark~\ref{R4.11}\eqref{R4.11.7}, assertion \eqref{C4.11.1} in Corollary~\ref{C4.11} must hold true.
Indeed, take $M:=1$, $x_1^k=x_2^k=u^k:=0$.
Then $f(x_1^k)=0$, $x_1^{*k}=\nabla f(x_1^k)=0$, $N_\Omega(x_2^k)=\{0\}$, and consequently, $x_2^{*k}=0$.
Hence, assertion \eqref{C4.11.1} in Corollary~\ref{C4.11} holds true.
\end{example}


\begin{example}
Let $X:=\R$, $\Omega:=\R_+$,
$f(x):=0$ if $x\le0$ and $f(x):=-x$ if $x>0$.
Note that $f$ is Lipschitz continuous.
Let $x^k:=0$ $(k\in\N)$.
Choose any $\{\eps^k\}\subset\R_+\setminus\{0\}$ such that $\eps^k\downarrow0$ and denote $\sigma_\eps:=\{\eps^k\}$.
Then $S_f(f(x^k)-\eps^k)=\{x\in\R\mid x\ge\eps^k\}$, $\zeta_{f,\Omega}(x^k,\eps^k)=\eps^k$,
$\al_{f,\Omega}(\sigma_x;\sigma_\eps) =\lim_{k\to+\infty}\frac{\eps^k}{\eps^k}=1$.
Since $\{\eps^k\}$ is arbitrary, $\sigma_x$ is not stationary for problem \eqref{P}.

Let $u^k:=-\sqrt{\eps^k}$ $(k\in\N)$ and $\sigma_u:=\{u^k\}$.
Then $\zeta_{f,\Omega}(u^k,\eps^k) =\eps^k+\sqrt{\eps^k}$, and consequently,
$\al_{f,\Omega}(\sigma_u;\sigma_\eps) =\lim_{k\to+\infty}\frac{\eps^k}{\eps^k+\sqrt{\eps^k}}=0$.
Hence, $\sigma_u$ is $\sigma_\eps$-sta\-tionary, and $\sigma_x$ is approximately $\sigma_\eps$-sta\-tionary for problem \eqref{P}.
In view of Remark~\ref{R4.11}\eqref{R4.11.7}, assertion \eqref{C4.11.1} in Corollary~\ref{C4.11} must hold true.
Indeed, take $M:=1$, $x_1^k:=-(\eps^k)^2$ and $x_2^k:=0$.
Then $f(x_1^k)=0$, $x_1^{*k}=\nabla f(x_1^k)=0$ and $x_2^{*k}:=0\in N_\Omega(x_2^k)$.
Conditions \eqref{C4.11-1} are trivially satisfied.
Hence, assertion \eqref{C4.11.1} in Corollary~\ref{C4.11} holds true.
\end{example}

Some sequential examples can be found in \cite[Section~5]{CuoKru}.

\section{Conclusions}
\label{conclusions}

We establish some refinements of the sequential concepts of extremality and stationarity of a collection of sets and the \emph{sequential extremal principle} introduced recently in \cite{CuoKru}.
The model adopted in the current paper extends and improves the one studied in \cite{CuoKru} along several lines as discussed in the Introduction.
In particular, the definitions adopted in the current paper are more specific in identifying the sequences relevant for the extremality and stationarity properties.
We discuss the properties corresponding to fixed sequences of translations. This line of research is new even in the conventional ‘at-a-point’ setting.
This leads to improved sequential dual necessary conditions with some additional restrictions on the dual variables and also sheds some new light on
extremality/stationarity settings in general, especially on the role of the vectors/sequences of translations.

We show that the sequential (as well as conventional) extremality and approximate stationarity properties possess certain stability, while the (non-approximate) stationarity does not.

In the dual statements, including the sequential extended extremal principle, we employ certain generalised separation conditions \GSs\ and \GSsal\ as well as a complementary primal-dual condition \PDs\ providing additional restrictions on the associated dual vectors.
The conditions are formulated in terms of Clarke (in general Banach spaces) or \Fr\ (in Asplund spaces) normals in an approximate/fuzzy form.
Imposing certain sequential normal compactness assumptions (which are automatically
satisfied in finite dimensions), one can formulate limiting versions of the conditions in terms of certain types of limiting normal cones.
The sequential extended extremal principle can be reformulated as a statement providing dual characterisations of the
absence of approximate stationarity. This property can
be interpreted as a kind of (sequential) transversality of collections of sets.
The last two topics are going to be explored in more detail elsewhere.

To illustrate the model, we consider a constrained minimisation problem in which the minimal value is not necessarily attained and deduce stronger than in \cite{CuoKru} sequential optimality and stationarity conditions for more general
types of stationary sequences.

\bigskip

\noindent{\bf Acknowledgements}
Part of the work was done during Alexander Kruger's stays at the Vietnam Institute for Advanced Study in Mathematics in Hanoi.
He is grateful to the institute for its hospitality and supportive environment.

\noindent{\bf Data availability. }
Data sharing is not applicable to this article as no datasets have been generated or analysed during the current study.

\section*{Declarations}

\noindent{\bf Conflict of interest.} The authors have no competing interests to declare that are relevant to the content of this article.

\addcontentsline{toc}{section}{References}
\bibliography{BUCH-kr,Kruger,KR-tmp}

\def\cprime{$'$} \def\cftil#1{\ifmmode\setbox7\hbox{$\accent"5E#1$}\else
  \setbox7\hbox{\accent"5E#1}\penalty 10000\relax\fi\raise 1\ht7
  \hbox{\lower1.15ex\hbox to 1\wd7{\hss\accent"7E\hss}}\penalty 10000
  \hskip-1\wd7\penalty 10000\box7} \def\cprime{$'$} \def\cprime{$'$}
  \def\cprime{$'$} \def\cprime{$'$} \def\cprime{$'$}
  \def\Dbar{\leavevmode\lower.6ex\hbox to 0pt{\hskip-.23ex \accent"16\hss}D}
  \def\cfac#1{\ifmmode\setbox7\hbox{$\accent"5E#1$}\else
  \setbox7\hbox{\accent"5E#1}\penalty 10000\relax\fi\raise 1\ht7
  \hbox{\lower1.15ex\hbox to 1\wd7{\hss\accent"13\hss}}\penalty 10000
  \hskip-1\wd7\penalty 10000\box7} \def\cprime{$'$}
\begin{thebibliography}{10}
\providecommand{\url}[1]{{#1}}
\providecommand{\urlprefix}{URL }
\expandafter\ifx\csname urlstyle\endcsname\relax
  \providecommand{\doi}[1]{DOI~\discretionary{}{}{}#1}\else
  \providecommand{\doi}{DOI~\discretionary{}{}{}\begingroup
  \urlstyle{rm}\Url}\fi

\bibitem{BorZhu05}
Borwein, J.M., Zhu, Q.J.: Techniques of Variational Analysis.
\newblock Springer, New York (2005)

\bibitem{BoyVan04}
Boyd, S.P., Vandenberghe, L.: Convex Optimization.
\newblock Springer Monographs in Mathematics. Cambridge University Press (2004)

\bibitem{BuiKru18}
Bui, H.T., Kruger, A.Y.: About extensions of the extremal principle.
\newblock Vietnam J. Math. \textbf{46}(2), 215--242 (2018).
\newblock \doi{10.1007/s10013-018-0278-y}

\bibitem{BuiKru19}
Bui, H.T., Kruger, A.Y.: Extremality, stationarity and generalized separation
  of collections of sets.
\newblock J. Optim. Theory Appl. \textbf{182}(1), 211--264 (2019).
\newblock \doi{10.1007/s10957-018-01458-8}

\bibitem{Cla83}
Clarke, F.H.: Optimization and Nonsmooth Analysis.
\newblock John Wiley \& Sons Inc., New York (1983).
\newblock \doi{10.1137/1.9781611971309}

\bibitem{Cuong3}
Cuong, N.D.: Dual characterizations of norm minimization problems.
\newblock Preprint, arXiv: \textbf{2601.08153} (2026)

\bibitem{Cuo26}
Cuong, N.D.: Primal and dual characterizations of sign-symmetric norms.
\newblock Positivity \textbf{30}(3), 35 (2026).
\newblock \doi{10.1007/s11117-026-01193-9}

\bibitem{CuoKru21.2}
Cuong, N.D., Kruger, A.Y.: Transversality properties: Primal sufficient
  conditions.
\newblock Set-Valued Var. Anal. \textbf{29}(2), 221--256 (2021).
\newblock \doi{10.1007/s11228-020-00545-1}

\bibitem{CuoKru25}
Cuong, N.D., Kruger, A.Y.: Generalized separation of collections of sets.
\newblock Optimization  (2026).
\newblock \doi{10.1080/02331934.2025.2562434}

\bibitem{CuoKru}
Cuong, N.D., Kruger, A.Y.: Sequential extremal principle and necessary
  conditions for minimizing sequences.
\newblock Optimization  (2026).
\newblock \doi{10.1080/02331934.2025.2534122}

\bibitem{CuoKruTha24}
Cuong, N.D., Kruger, A.Y., Thao, N.H.: Extremality of families of sets.
\newblock Optimization \textbf{73}(12), 3593--3607 (2024).
\newblock \doi{10.1080/02331934.2024.2385656}

\bibitem{CuoKruTha25}
Cuong, N.D., Kruger, A.Y., Thao, N.H.: Extremality of families of sets and
  set-valued optimization.
\newblock Set-Valued Var. Anal. \textbf{33}(2), art. no. 21 (2025).
\newblock \doi{10.1007/s11228-025-00751-9}

\bibitem{FabMor02}
Fabian, M., Mordukhovich, B.S.: Separable reduction and extremal principles in
  variational analysis.
\newblock Nonlinear Anal., Ser. A: Theory Methods \textbf{49}(2), 265--292
  (2002).
\newblock \doi{10.1016/S0362-546X(01)00107-9}

\bibitem{KimNguPha25}
Kim, D.S., Nguyen, M.T., Pham, T.S.: Subdifferentials at infinity and
  applications in optimization.
\newblock Math. Program., Ser. A \textbf{214}(1-2), 409--440 (2025).
\newblock \doi{10.1007/s10107-024-02187-9}

\bibitem{Kru85.1_}
Kruger, A.Y.: Generalized differentials of nonsmooth functions and necessary
  conditions for an extremum.
\newblock Siberian Math. J. \textbf{26}, 370--379 (1985)

\bibitem{Kru98}
Kruger, A.Y.: About extremality of systems of sets.
\newblock Dokl. Nats. Akad. Nauk Belarusi \textbf{42}(1), 24--28 (1998).
\newblock In Russian. Available from:
  https://asterius.federation.edu.au/akruger

\bibitem{Kru02}
Kruger, A.Y.: Strict {$(\epsilon,\delta)$}-subdifferentials and extremality
  conditions.
\newblock Optimization \textbf{51}(3), 539--554 (2002).
\newblock \doi{10.1080/0233193021000004967}

\bibitem{Kru03}
Kruger, A.Y.: On {F}r\'{e}chet subdifferentials.
\newblock J. Math. Sci. (N.Y.) \textbf{116}(3), 3325--3358 (2003).
\newblock \doi{10.1023/A:1023673105317}

\bibitem{Kru04}
Kruger, A.Y.: Weak stationarity: eliminating the gap between necessary and
  sufficient conditions.
\newblock Optimization \textbf{53}(2), 147--164 (2004)

\bibitem{Kru05}
Kruger, A.Y.: Stationarity and regularity of set systems.
\newblock Pac. J. Optim. \textbf{1}(1), 101--126 (2005)

\bibitem{Kru06}
Kruger, A.Y.: About regularity of collections of sets.
\newblock Set-Valued Anal. \textbf{14}(2), 187--206 (2006).
\newblock \doi{10.1007/s11228-006-0014-8}

\bibitem{Kru06.2}
Kruger, A.Y.: Stationarity and regularity of real-valued functions.
\newblock Appl. Comput. Math. \textbf{5}(1), 79--93 (2006)

\bibitem{Kru09}
Kruger, A.Y.: About stationarity and regularity in variational analysis.
\newblock Taiwanese J. Math. \textbf{13}(6A), 1737--1785 (2009).
\newblock \doi{10.11650/twjm/1500405612}

\bibitem{KruMor80}
Kruger, A.Y., Mordukhovich, B.S.: Extremal points and the {E}uler equation in
  nonsmooth optimization problems.
\newblock Dokl. Akad. Nauk BSSR \textbf{24}(8), 684--687 (1980).
\newblock In Russian. Available from:
  https://asterius.federation.edu.au/akruger

\bibitem{KruTha13}
Kruger, A.Y., Thao, N.H.: About uniform regularity of collections of sets.
\newblock Serdica Math. J. \textbf{39}(3-4), 287--312 (2013)

\bibitem{Mor00}
Mordukhovich, B.S.: An abstract extremal principle with applications to welfare
  economics.
\newblock J. Math. Anal. Appl. \textbf{251}(1), 187--216 (2000)

\bibitem{Mor06.1}
Mordukhovich, B.S.: Variational Analysis and Generalized Differentiation. {I}:
  {B}asic Theory.
\newblock Springer, Berlin (2006).
\newblock \doi{10.1007/3-540-31247-1}

\bibitem{Mor06.2}
Mordukhovich, B.S.: Variational Analysis and Generalized Differentiation. {II}:
  {{A}pplications}.
\newblock Springer, Berlin (2006)

\bibitem{MorNam22}
Mordukhovich, B.S., Nam, N.M.: Convex Analysis and Beyond. Volume I: Basic
  Theory.
\newblock Springer Series in Operations Research and Financial Engineering.
  Springer (2022).
\newblock \doi{10.1007/978-3-030-94785-9}

\bibitem{MorSha96}
Mordukhovich, B.S., Shao, Y.: Extremal characterizations of {A}splund spaces.
\newblock Proc. Amer. Math. Soc. \textbf{124}(1), 197--205 (1996)

\bibitem{MorTreZhu03}
Mordukhovich, B.S., Treiman, J.S., Zhu, Q.J.: An extended extremal principle
  with applications to multiobjective optimization.
\newblock SIAM J. Optim. \textbf{14}(2), 359--379 (2003)

\bibitem{NguPha24}
Nguyen, M.T., Pham, T.S.: Clarke's tangent cones, subgradients, optimality
  conditions, and the {L}ipschitzness at infinity.
\newblock SIAM J. Optim. \textbf{34}(2), 1732--1754 (2024).
\newblock \doi{10.1137/23M1545367}

\bibitem{NguPha26}
Nguyen, M.T., Pham, T.S.: The radius of metric regularity at infinity.
\newblock Positivity \textbf{30}(2), 19 (2026).
\newblock \doi{10.1007/s11117-026-01175-x}

\bibitem{Phe93}
Phelps, R.R.: Convex Functions, Monotone Operators and Differentiability, 2nd
  edn.
\newblock Springer-Verlag, Berlin (1993).
\newblock \doi{10.1007/978-3-662-21569-2}

\bibitem{ZheNg05.2}
Zheng, X.Y., Ng, K.F.: The {F}ermat rule for multifunctions on {B}anach spaces.
\newblock Math. Program. \textbf{104}(1), 69--90 (2005).
\newblock \doi{10.1007/s10107-004-0569-9}

\bibitem{ZheNg11}
Zheng, X.Y., Ng, K.F.: A unified separation theorem for closed sets in a
  {B}anach space and optimality conditions for vector optimization.
\newblock SIAM J. Optim. \textbf{21}(3), 886--911 (2011).
\newblock \doi{10.1137/100811155}

\end{thebibliography}
\bibliographystyle{spmpsci}
\end{document}